\documentclass[french]{article} 
\usepackage[utf8]{inputenc}
\usepackage[T1]{fontenc}
\usepackage{lmodern}
\usepackage[a4paper,vmargin={3.4cm,3.4cm},hmargin={2.7cm,2.7cm}]{geometry}

\usepackage[english]{babel}

\usepackage{fullpage}

\usepackage{amssymb,amsmath,mathtools,amsthm,dsfont,stmaryrd, bm}
\usepackage{graphicx}

\usepackage[all]{xy}
\usepackage{xcolor}
\usepackage{todonotes}

\newcommand{\zi}{z}

\newtheorem{theorem}{Theorem}
\newtheorem{prop}[theorem]{Proposition}
\newtheorem{corollary}[theorem]{Corollary}
\newtheorem{lemma}[theorem]{Lemma}

\newtheorem{remark}[theorem]{Remark}
\newtheorem{defi}[theorem]{Definition}

\newcommand{\E}{\mathbb{E}}

\newcommand{\N}{\mathbb{N}}

\renewcommand{\P}{\mathbb{P}}

\newcommand{\T}{\mathcal{T}}
\newcommand{\Tf}{T}

\newcommand{\Z}{\mathbb{Z}}
\renewcommand{\dir}{\mathop{dir}}

\renewcommand{\epsilon}{\varepsilon}
\renewcommand{\phi}{\varphi}

\newcounter{numeroexo}

\begin{document}

\title{\LARGE Random sequential nearest-neighbor coloring on trees}

\author{Anne-Laure Basdevant\footnote{LPSM, Sorbonne Universit\'e, France, {\it anne.laure.basdevant@normalesup.org}}, Arvind Singh\footnote{CNRS and LMO, Universit\'e Paris-Saclay, France
 {\it arvind.singh@universite-paris-saclay.fr}}}
\date{\today}

\selectlanguage{english}

\maketitle

\vspace{-0.5cm}

\begin{figure}[h!]
    \centering
\includegraphics[width=11cm]{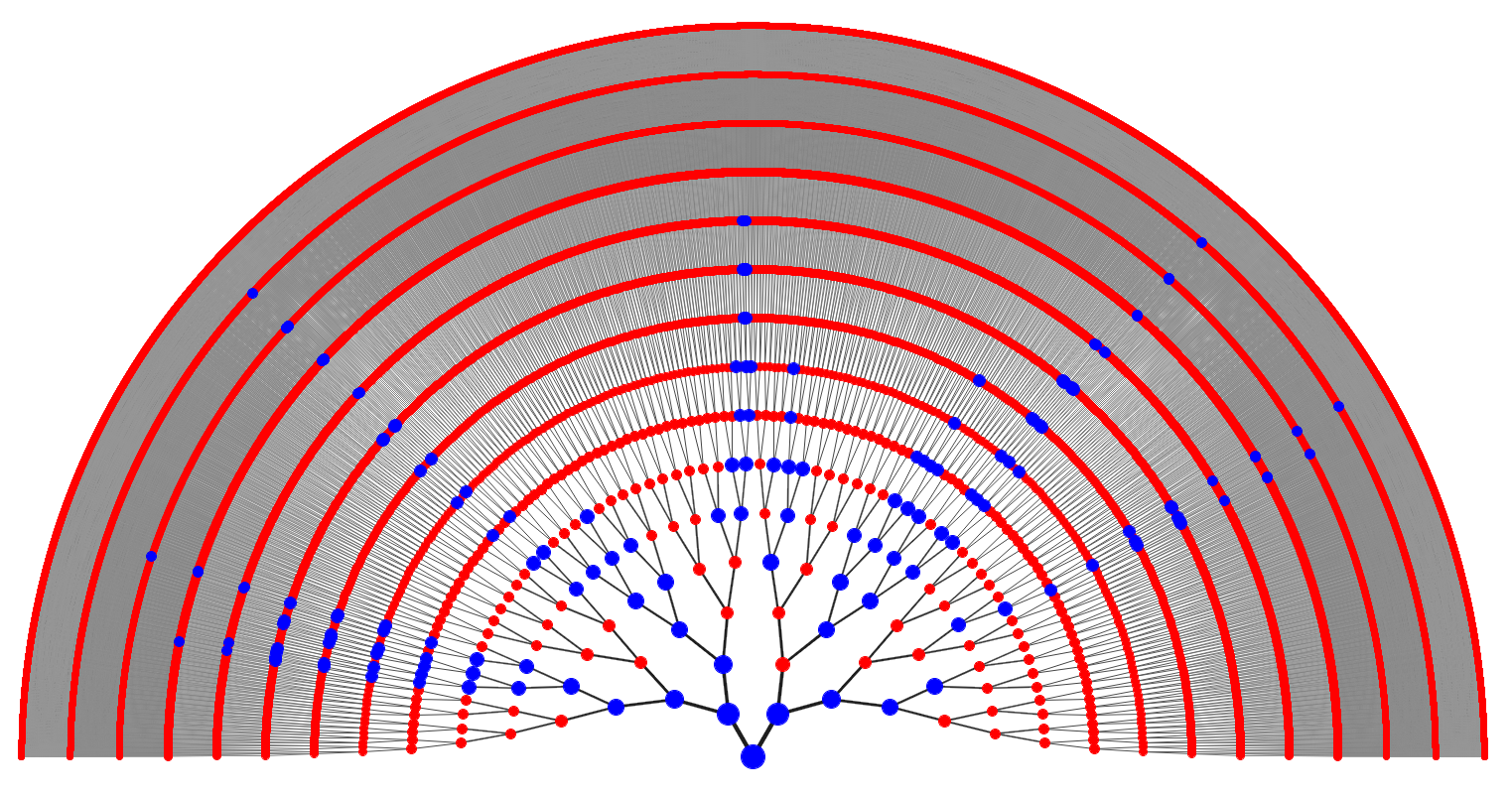}
\caption{\small{Simulation of the coloring process on a binary tree of height $15$. All the leaves are seeds initially colored in red while the root is the seed of the blue cluster.}} \label{fig:abstract}
\end{figure}

\begin{abstract}
We study a nearest-neighbor coloring process in which vertices are revealed in random order and inherit the color of the closest vertex revealed before them. This model is a discrete analogue of coloring processes previously studied by Preater \cite{Preater2009} and Aldous \cite{Aldous18} in Euclidean spaces. We focus here on regular trees and analyze the associated genealogy of color inheritance. In contrast with the Euclidean case, the genealogical graph on an infinite regular tree is not connected: it has infinitely many infinite one-ended components, each with a distinct asymptotic direction, while every vertex has only finitely many descendants. We also describe how this structure is modified in the presence of finitely many initial seeds. Finally, we study local limits of the coloring on finite regular trees as their height tends to infinity, for two natural seed configurations: two fixed seeds, and one blue seed at the root with red seeds at the leaves.
\end{abstract}

\section{Introduction}

We study a random coloring procedure on graphs, in which vertices are colored sequentially at random, each vertex taking, at the time it is picked, the color of one of its closest already colored vertices. This model can be interpreted as a growth process driven by local competition between different species (\emph{i.e.}, colors). Alternatively, it may also be interpreted as a variant of the classical voter model where new settlers take the political opinion of one of their already settled neighbors.
\medskip

Sequential nearest-neighbor rules of this kind have appeared in two closely related forms. In the uncolored version, each arriving point is linked to its closest predecessor, giving the on-line nearest-neighbor graph, or random nearest-neighbor tree, whose geometry has been studied in several Euclidean settings \cite{PenroseWade2008,PenroseWade2009,LichevMitsche2024,Casse2025}. In the colored version, introduced by Preater and Aldous and further studied in \cite{Preater2009,Aldous18,BasdevantBlancCurienSingh}, the focus is instead on the random partition generated by color inheritance. More precisely, in the latter Euclidean model, the authors consider the unit cube $[-1,1]^d$. Initially, the origin is colored in blue and the boundary of the cube in red. Subsequently, independent, uniformly sampled random points fall in $[-1,1]^d$ and, upon arrival, each point takes the color of the nearest point that has appeared so far. This procedure ultimately creates a random coloring of the unit cube\footnote{This construction produces only a countable set of colored sample points but can be subsequently extended to a coloring of the whole cube by assigning to each location the color of the points in its infinitesimal neighborhood (with an additional ``gray'' interface corresponding to accumulation points from both colors).}. One of the main results in \cite{Preater2009, BasdevantBlancCurienSingh} is that there exists a.s. an open region containing the origin that is fully blue, \emph{i.e.}, all red vertices remain at a positive distance from the initial blue vertex. On the other hand, the converse question of whether there may exist blue vertices arbitrarily close to the red boundary with positive probability is more delicate and still remains open.

\medskip
In contrast with these Euclidean works, the present paper deals with a discrete setting in which the underlying graph is a regular tree. This change of geometry leads to rather different phenomena and makes the genealogical structure of the process a natural object of study. Before specializing to trees, let us first describe the model on an arbitrary finite connected graph $G=(V,E)$. Given a set of initially colored vertices called \emph{seeds} and a uniformly random permutation of the non-seed vertices (which encodes the order in which the vertices are revealed), the coloring procedure is defined as follows. When a vertex $v$ is revealed, it takes the color of one of the already colored vertices that minimize the graph distance to $v$; if several such vertices exist, one chooses uniformly among them. At the end of the procedure, all vertices are colored.

\medskip

The same model can equivalently be formulated in terms of arrival times, in a way that naturally extends to infinite graphs: one assigns to each vertex $v$ a time $\tau_v \geq 0$, interpreted as its arrival time, and reveals the vertices in increasing order of these times. On a finite graph, if the variables $(\tau_v)_{v\in V}$ are i.i.d.\ with a continuous distribution, then their increasing order induces a uniform random permutation of the vertices and we recover the previous construction of the coloring. On an infinite graph, there is in general no way to reveal the vertices one after another in increasing order of their arrival times, since the collection $(\tau_v)$ need not admit a smallest element. Nevertheless, the variables $(\tau_v)$ still encode the relative order of arrival of the vertices and make it possible to reconstruct the genealogy by a local backward exploration. Vertices $s$ with arrival time $\tau_s = 0$ are declared to be seeds, have no ancestor, and are assigned a fixed color. If $v$ has positive arrival time, we write $d(\cdot,\cdot)$ for the graph distance and set
\begin{equation}\label{eq:radius}
r(v):=\min\bigl\{d(v,w): w\in V,\, \tau_w<\tau_v\bigr\}.
\end{equation}
The parent of $v$ is defined as the vertex at distance $r(v)$ with an arrival time strictly smaller than $\tau_v$ (chosen uniformly at random if there is more than one). Iterating this parent relation yields an ancestral line for each vertex and, in turn, a genealogical forest $\mathcal{G}$ on the set of vertices $V$. Each connected component of $\mathcal{G}$ is a tree with all its vertices of the same color.

\medskip

In the present paper, we focus on the case where the graph $G$ is a regular tree. Our main purpose is to study the genealogy generated by the coloring procedure. In the infinite-volume setting, we show that this genealogy is a forest with infinitely many connected components and investigate its geometric structure. We then return to finite regular trees and derive from the infinite-volume picture several consequences for the corresponding finite-volume colorings and their local limits.

\subsection{Main results}

We now introduce some notation and state the main results of the paper.

\subsubsection{Structure of the genealogical graph on a regular infinite tree.}

We consider the infinite regular tree $\T$ in which every vertex has $K+1$ neighbors, where $K \ge 2$ is fixed. We fix a vertex $o \in \T$, which we call the root. For $u,v \in \T$, we write $d(u,v)$ for the graph distance between $u$ and $v$, and $|v| := d(o,v)$ for the height of the vertex $v$. For $v\in \T$ and $r\in \N$, we denote by $B(v,r)$ (resp.\ $S(v,r)$) the ball (resp.\ sphere) of radius $r$ and center $v$.

\medskip

We denote by $\partial \T$ the boundary at infinity of $\T$, that is, the set of infinite rays\footnote{We call a ray a geodesic starting from the root, \emph{i.e.}, a non-backtracking path $o = v_0, v_1,v_2, \ldots$ where $v_{i+1}$ is a neighbor of $v_i$ and $v_{i+1} \neq v_{i-1}$.} starting from the root. Finally, for $v \in \T$, we denote by $\T_v$ the descendant subtree rooted at $v$, namely the induced subtree consisting of all vertices $w \in \T$ such that the geodesic path from $o$ to $w$ passes through $v$.

\medskip

We first consider the genealogical graph $\mathcal{G}$ associated with the coloring process on $\T$ without seeds. More precisely, we fix an i.i.d.\ sequence of arrival times $(\tau_v,\; v \in \T)$ sampled according to a continuous distribution\footnote{The choice of the distribution does not matter provided it has no atoms. In practice, it will be convenient to choose either a uniform or exponential distribution.} on $(0,\infty)$. Therefore, all vertices have a positive arrival time (there are no initial seeds) and, for each vertex $v$, we can define its parent $p(v)$ by the exploration procedure explained in \eqref{eq:radius}. This defines an oriented graph $\mathcal{G}$ on $\T$, called the \emph{genealogical graph}, whose oriented edges are of the form $(v,p(v))$. We call its connected components the \emph{genealogical components}. By construction, every connected component of $\mathcal G$ is necessarily an oriented tree\footnote{By convention, edges in the genealogical graph are always oriented from a vertex to its parent.}.
\medskip

We stress that $\mathcal{G}$ has the same vertex set as $\T$ but is not compatible with the graph structure of $\T$ since $p(v)$ is not, in general, a neighbor of $v$. In particular, the genealogical components are trees but are not subtrees of $\T$. 

\medskip

For each vertex $v \in \T$, we denote by
\begin{equation}\label{eq:def_path}
    \Gamma(v) := (v,p(v),p^{(2)}(v),\dots)
\end{equation}
its ancestral path, where $p^{(k)}(v)$ is the $k$-th iterate of the parent map. We observe that two vertices $u$ and $v$ belong to the same component if and only if their ancestral paths eventually coalesce, i.e.\ there exist integers $m,n \ge 0$ such that
\[
p^{(m+k)}(u)=p^{(n+k)}(v)\qquad \text{for all }k\ge 0.
\]
For $v \in \T$, we also define the set of all its descendants:
\[
D(v):=\{w\in \T:\; v\in \Gamma(w)\}.
\]

The following theorem describes the structure of $\mathcal{G}$.

\begin{theorem}[\textbf{The genealogical graph without seeds}]\label{thm:GwithoutSeeds}
Consider the coloring process on $\T$ without seeds and let $\mathcal G$ denote the associated genealogical graph. Then the following assertions hold almost surely:
\begin{enumerate}
\item The genealogical graph $\mathcal G$ has infinitely many connected components. Furthermore, denoting by $N_n$ the number of connected components of $\mathcal G$ intersecting the sphere $S(o,n)$, there exists a constant $c>0$ (depending only on $K$) such that
    $$\liminf_{n\to \infty} \frac{N_n}{\sharp S(o,n)}\ge c \qquad \mbox{ a.s.}.$$
    \item Every connected component of $\mathcal G$ is an infinite one-ended oriented tree, i.e.\ for every $v \in \T$, its ancestral path $\Gamma(v)$ is infinite while its set of descendants $D(v)$ is finite and satisfies
    \[
    \E[\sharp D(v) \mid \tau_v]<\infty \qquad \mbox{ a.s.}.
    \]

    \item Every connected component $\mathcal C$ of $\mathcal G$ admits a limit direction, i.e.\ there exists an infinite ray $\xi=(w_i)_{i\ge 1}\in \partial \T$ such that, for every vertex $v$ in the component and every $i \ge 1$, the vertices of $\Gamma(v)$ eventually belong to the subtree $\T_{w_i}$. We denote this limit by
    \[
    \dir \mathcal C \in \partial \T.
    \]

    \item For every vertex $v \in \T$, there exists, almost surely, a connected component $\mathcal C$ such that $\dir \mathcal C \in \partial \T_v$. 

    \item If $\mathcal C,\mathcal C'$ are two distinct connected components of $\mathcal G$, then
    \[
    \dir \mathcal C \neq \dir \mathcal C' \qquad \text{a.s.}.
    \]
    
\end{enumerate}
\end{theorem}

\newpage

\begin{remark} 
\begin{enumerate}
    \item The fact that the genealogical graph $\mathcal{G}$ has infinitely many connected components is related to the exponential growth of the underlying tree $\T$. If one considers instead the lattice $G = \mathbf{Z}^d$ for $d\geq 2$ as the base graph, then the associated genealogical graph is a.s. connected. This result can be proved using arguments similar to those used by Aldous to study the coalescence time for the "coupled EA process" (Section~$3$ of \cite{Aldous18}). The arguments of \cite{Aldous18} are for the continuous Poissonian model in the plane but can be extended, with some adjustments, to the lattice case for any dimension $d\geq 2$.
    \item Here, we are working with a base graph $\T$ that is a tree, but the acyclic nature of the graph does not seem fundamental. We believe that the description of $\mathcal{G}$ given in Theorem \ref{thm:GwithoutSeeds} should remain true for more general vertex-transitive graphs with exponential growth. We also expect that Theorem \ref{thm:GwithoutSeeds} should admit a continuous analogue for the Poissonian coloring process in hyperbolic space. This is the subject of work in progress (see Figure \ref{fig:hyper} for a simulation of the coloring process on the Poincaré hyperbolic disk).
\end{enumerate}
\end{remark}

\begin{figure}
\centering
\includegraphics[height=9cm]{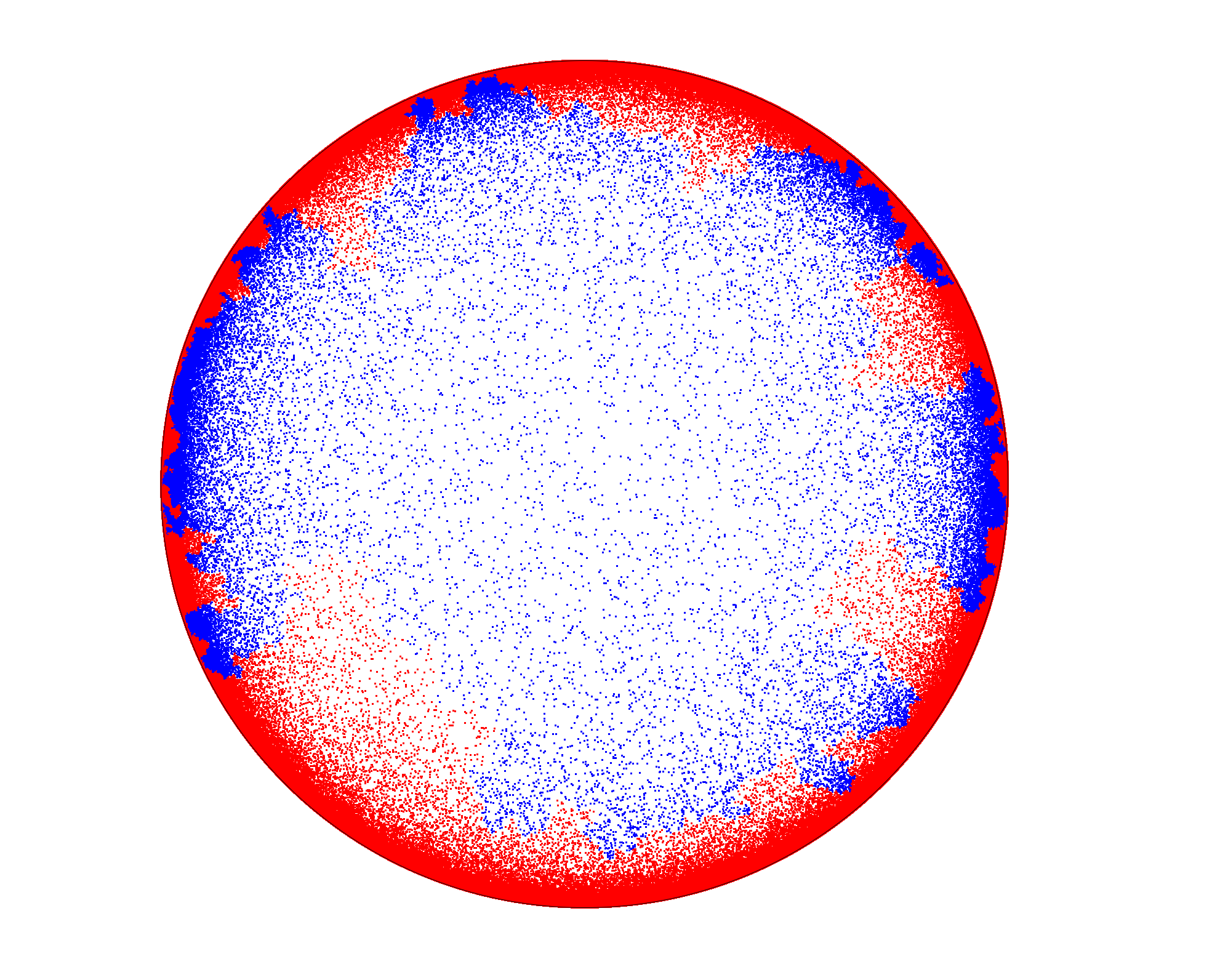}
\caption{\small{Coloring on the Poincaré hyperbolic disk starting from a single blue point at the origin and a red ideal boundary. Simulation for $10^7$ points.}} \label{fig:hyper}
\end{figure}

\begin{theorem}[\textbf{The genealogical graph with a finite number of seeds}]\label{thm:GwithSeeds}
Let $\mathcal G$ denote the genealogical graph of the coloring process on $\T$ without seeds. Write the decomposition of $\mathcal G$ into its connected components as:
$$
\mathcal G = \bigsqcup_{i \in I} \mathcal{C}_i.
$$
Fix $M\geq 1$ and vertices $\tilde{s}_1, \ldots \tilde{s}_M \in \T$ chosen as seeds. Let $\tilde{\mathcal{G}}$ denote the genealogical graph constructed\footnote{In this coupled construction of $\mathcal G$ and $ \tilde{\mathcal{G}}$, we also require ties (when there are multiple possible ancestors) to be broken in the same way, see Section \ref{sec:constructionW} for more details.} with the same i.i.d.\ sequence of arrival times $(\tau_v,\, v\in \T)$ except that now $\tau_{\tilde{s}_j}=0$ for all $1\le j\le M$. Then the decomposition of $\tilde{\mathcal{G}}$ into connected components satisfies
$$
\tilde{\mathcal G} = \bigsqcup_{i \in I} \tilde{\mathcal{C}}_i \sqcup \bigsqcup_{j = 1}^M\tilde{\mathcal{U}}_j,
$$
where 
\begin{enumerate}
    \item For each $i\in I$, $\tilde{\mathcal{C}}_i$ is an infinite oriented subtree of $\mathcal{C}_i$. In particular, we have $\dir \tilde{\mathcal{C}}_i = \dir \mathcal{C}_i$.

    \item For each $j\in \{1,\ldots, M\}$, $\tilde{\mathcal{U}}_j$ is an oriented tree rooted at $\tilde{s}_j$. Moreover, 
    \begin{enumerate}
        \item For $v \in \tilde{\mathcal{U}}_j$, $v \neq \tilde{s}_j$, the set $\tilde{D}(v)$ of descendants of $v$ in $\tilde{\mathcal{U}}_j$ is finite with $\E[\sharp\tilde{D}(v) \mid \tau_v] < \infty$.
\item The component $\tilde{\mathcal{U}}_j$ is infinite if and only if $\tilde{s}_j$ has infinite degree. Furthermore, we have 
$$
\P(\mbox{$\tilde{s}_j$ has infinite degree in $\tilde{\mathcal{U}}_j$})=
\begin{cases}
    1 & \mbox{if $\{v\in \T, d(v,\tilde{s}_j)=\min_{i\le M} d(v,\tilde{s}_i)\}$ is infinite}\\
    0 & \mbox{otherwise}
\end{cases}.
$$
\end{enumerate}    
\end{enumerate}
\end{theorem}

\begin{remark} 
We note that if there are $M\le K+1$ seeds on a $(K+1)$-regular tree, then necessarily all the seeds have infinite degree in the genealogical graph. 
\end{remark}

\subsubsection{Coloring on finite regular trees and local limits}

In the second part of this paper, we describe the coloring process on large finite regular trees and study its limit as the height of the tree tends to infinity. For $\ell \geq 0$, let
$$
\T^{(\ell)} := \T \cap B(o,\ell)
$$
denote the finite regular tree of height $\ell$.

\medskip

\noindent\textbf{Coloring with seeds at the root and at the leaves.}
First, we consider the coloring process on $\T^{(\ell)}$, where seeds are placed at the root and at every leaf of $\T^{(\ell)}$ (that is, every vertex at height $\ell$). This setting mimics, in some sense, the model studied in the continuous setting on $[-1,1]^d$, with a blue point at the origin and the boundary of the cube colored red. Thus, here the seed at the root propagates the blue color to its descendants, whereas all seeds at the leaves propagate the same red color to their descendants (see Figure \ref{fig:abstract} for an illustration of this process). We denote by $B^{(\ell)}$ (resp.\ $R^{(\ell)}$) the final sets of blue (resp.\ red) vertices.

\medskip

Note that for $\ell<\ell'$, there is a natural coupling for the coloring processes on $\T^{(\ell)}$ and $\T^{(\ell')}$ obtained by considering the same ordering of the vertices (\emph{i.e.}, the same arrival times) and simply ignoring vertices $v$ with height $|v| > \ell$ when coloring $\T^{(\ell)}$. For the particular seed configuration considered here, this coupling is monotone in the sense that
\begin{eqnarray}\label{eq:coupling}
B^{(\ell)}\subset B^{(\ell')}\quad \hbox{a.s. for all $\ell<\ell'$.}    
\end{eqnarray}
We are particularly interested in the following quantities:
\begin{equation}\label{eq:M_N_m_n}
    M_{\ell}:=\max\{|v|:\; v\in B^{(\ell)}\},
\qquad
m_{\ell}:=\min\{|v|:\; v\in R^{(\ell)}\},
\end{equation}
that is, the maximal height reached by the blue region and the minimal height reached by the red region. By \eqref{eq:coupling}, the sequences $(m_\ell)_{\ell\ge 1}$ and $(M_\ell)_{\ell\ge 1}$ are both non-decreasing with $\ell$.

\begin{theorem}\label{theo:main} Consider the coloring process on $\T^{(\ell)}$ with a blue seed at the root and red seeds at the leaves constructed with the coupling described above. Then the coloring process converges a.s.\ in the local topology to the random coloring of $\T$ obtained by coloring the genealogical component of the root in blue and all other components in red. Furthermore, the sequence $(M_\ell)_{\ell\ge 1}$ converges a.s. to infinity and satisfies
\begin{equation}\label{eq:max}
  \lim_{\ell\to \infty} \P(M_{\ell}=\ell-1)=1.  
\end{equation}
The sequence $(m_\ell)_{\ell\ge 1}$ converges a.s. to a finite random variable $m_\infty$ such that
\begin{equation}\label{eq:min}
   \P(m_\infty=1)=\inf_{\ell\ge 1} \P(m_\ell=1)>0.  
\end{equation} 
\end{theorem}

The a.s. convergence of the coloring process to the limit coloring of $\T$ described above is a soft result that follows from the coupling and the description of the genealogical graph $\mathcal{G}$ given in Theorem \ref{thm:GwithoutSeeds}. The more interesting (and more difficult) part of the theorem is \eqref{eq:max}, which, somewhat counterintuitively, states that the blue region not only grows to infinity but, in fact, reaches maximum height $\ell-1$ with high probability as $\ell$ tends to infinity. This result is in stark contrast with the behavior of the coloring process in Euclidean space, where it is known that, with positive probability, the blue region does not reach the boundary\footnote{The open question mentioned at the beginning of the introduction is whether this occurs with probability~$1$, see open Question 1.4 in \cite{BasdevantBlancCurienSingh} and the discussion preceding Corollary 1 in \cite{Preater2009} for more details.}.

\medskip 

Statement \eqref{eq:min} of Theorem~\ref{theo:main} shows that, although the distance from the root to the red seeds increases with the height of the tree, the coloring remains non-trivial around the root. Again, this phenomenon contrasts with the Euclidean setting: if one considers the Poissonian coloring in a ball with red boundary conditions and a blue seed at the origin, then letting the radius tend to infinity (together with an appropriate scaling) leads to a trivial limit in which the coloring becomes entirely blue. On the infinite tree $\T$, however, the limit remains non-trivial because the red seeds do not disappear; instead, they are pushed to infinity, \emph{i.e.}, to the ideal boundary of the tree. This phenomenon is reminiscent of situations arising in hyperbolic geometry, where infinite-volume limits of random structures may naturally live on the ideal boundary of the space. A related example is the recently introduced model of ideal Poisson--Voronoi tessellations \cite{dAchilleEtAl23}.

\bigskip

\noindent\textbf{Two-seed coloring.} We now study the coloring process on the finite tree with two seeds. Fix another vertex $o' \in \T$ and for $\ell > d(o,o')$, consider the coloring process on $\T^{(\ell)}$ with a blue seed at the root and a red seed at $o'$. As usual, we denote by $\tilde{B}^{(\ell)}$ (resp.\ $\tilde{R}^{(\ell)}$) the set of blue (resp.\ red) vertices.
We describe the local limit in distribution of the process $(\tilde{B}^{(\ell)},\tilde{R}^{(\ell)})$ in terms of the genealogical graph associated with the coloring process on the infinite tree $\T$ with two seeds. We first introduce some notation.

\medskip

Let $\tilde{\mathcal G}$ denote the genealogical graph of the coloring process on the infinite tree $\T$ with two seeds at $o$ and $o'$. Theorem~\ref{thm:GwithSeeds} states that the decomposition of $\tilde{\mathcal G}$ into connected components satisfies
$$
\tilde{\mathcal G}
=
\bigsqcup_{i \in I} \tilde{\mathcal C}_i
\sqcup
\tilde{\mathcal U}_o
\sqcup
\tilde{\mathcal U}_{o'},
$$
where $\tilde{\mathcal U}_o$ (resp.\ $\tilde{\mathcal U}_{o'}$) is the genealogical component of $o$ (resp.\ $o'$), and each genealogical component $\tilde{\mathcal C}_i$ is an infinite oriented tree with asymptotic direction $\dir \tilde{\mathcal C}_i$. Observe that, for any asymptotic direction in $\partial \T$, we can determine whether it is asymptotically closer 
to $o$, closer to $o'$, or asymptotically at equal distance from $o$ and $o'$.

\medskip

\noindent Now, consider the following coloring procedure:
\begin{itemize}
    \item Color all vertices of $\tilde{\mathcal U}_o$ in blue and all vertices of $\tilde{\mathcal U}_{o'}$ in red.
    
    \item For the clusters $\tilde{\mathcal C}_i$ such that their asymptotic direction $\dir \tilde{\mathcal C}_i$ is asymptotically closer to $o$ (resp.\ to $o'$), color all their vertices in blue (resp.\ red).
    
    \item For each cluster $\tilde{\mathcal C}_i$ such that its asymptotic direction $\dir \tilde{\mathcal C}_i$ is at equal distance from $o$ and $o'$, pick an independent Bernoulli random variable $X_i$ with parameter $\frac{1}{2}$ and color the cluster in blue if $X_i=1$ and red otherwise.
\end{itemize}
This procedure yields a random coloring process of $\T$ in two colors; we denote by $\chi_{o'}$ its distribution.

\begin{theorem}\label{theo:twoseed} Consider the coloring process on the finite tree $\T^{(\ell)}$ with two seeds located at $o$ and $o'$. Then the coloring $(\tilde{B}^{(\ell)},\tilde{R}^{(\ell)})$ on $\T^{(\ell)}$ converges in distribution in the local topology to $\chi_{o'}$ on $\T$.
\end{theorem}

In fact, if we construct the process $(\tilde{B}^{(\ell)},\tilde{R}^{(\ell)})$ and $\tilde{\mathcal G}$ using the same sequence of arrival times $(\tau_v)_{v \in \T}$, then the previous result can be strengthened to yield almost sure convergence of $(\tilde{B}^{(\ell)},\tilde{R}^{(\ell)})$ outside the clusters whose asymptotic directions are at equal distance from $o$ and $o'$ (but not for the cluster at equal distance). In particular, we also obtain the following proposition.

\begin{prop}\label{prop:twoseed}
Let $o'$ be a vertex of $\T$ such that $d(o,o')$ is odd. For $\ell > d(o,o')$, consider the coloring process on the finite tree $\T^{(\ell)}$ with two seeds located at $o$ and $o'$, and assume that all processes are constructed using the same sequence of arrival times. Then $
(\tilde{B}^{(\ell)},\tilde{R}^{(\ell)})_{\ell > d(o,o')}
$
converges almost surely in the local topology to some random sets
$
(\tilde{B}^{(\infty)},\tilde{R}^{(\infty)}).
$
Moreover:
\begin{itemize}
    \item[\textup{(a)}]
    $\tilde{B}^{(\infty)} \cap \{v \in \T : d(v,o)>d(v,o')\}$ is a.s. finite;
    
    \item[\textup{(b)}]
    $\tilde{R}^{(\infty)} \cap \{v \in \T : d(v,o)<d(v,o')\}$ is a.s. finite.
\end{itemize}
\end{prop}

The assumption that $d(o,o')$ is odd is required here to ensure that no vertex of $\T$ can be at the same distance from $o$ and $o'$. This implies that the color of a vertex $v$ eventually stabilizes according to whether it belongs to a cluster of $\tilde{\mathcal G}$ whose asymptotic direction is closer to $o$ or to $o'$.

We will, in fact, also show that if $d(o,o')$ is even, the sequence $(\tilde{B}^{(\ell)},\tilde{R}^{(\ell)})_{\ell > d(o,o')}$ still converges outside the set consisting of vertices that are at equal distance from $o$ and $o'$, and the two finiteness statements above remain valid.
Thus, from a macroscopic perspective, (a) and (b) tell us that the coloring of $\T^{(\ell)}$ with two seeds is ``almost'' deterministic: apart from finitely many vertices, all those that are closer to $o$ are colored blue, while those that are closer to $o'$ are colored red. Simulations also suggest that the same phenomenon occurs for a similar coloring process on the hyperbolic plane.

\subsection*{Outline of the paper}

In Section \ref{s:genealogicalprocess}, we give the precise construction of the ancestral path $\Gamma(v)$ and of an associated ancestral process $(V_n(v),W_n(v))$. This includes the uniform increasing indexings used to break ties and an equivalent construction in terms of the inter-arrival times of a lower record process. We then derive the preliminary estimates on this process, in particular the growth estimates for the jumps and for the height of the ancestral path (Proposition \ref{prop:V_nconditionnelle} and Corollary \ref{coro:V_nenn2}).

\medskip

Section \ref{s:Gwithoutseed} is devoted to the proof of Theorem~\ref{thm:GwithoutSeeds}. We first prove that every genealogical component has an asymptotic direction. We then use Proposition~\ref{prop:keyresult}, which shows that an ancestral path may avoid an entire descendant subtree with positive probability, to prove that $\mathcal{G}$ has infinitely many components and that every descendant sector contains the direction of some component. The fact that two distinct components have distinct directions is proved by introducing a coupled exploration of two ancestral paths, encoded by a lower 2-record process. Finally, we prove that every vertex has only finitely many descendants by a combinatorial estimate on possible descendant paths.

\medskip 

In Section \ref{s:Gwithseed}, we prove Theorem~\ref{thm:GwithSeeds}. The argument is based on a coupling between the genealogical graph without seeds and the genealogical graph with a finite set of seeds: the ancestral path of a vertex is changed exactly when its exploration discovers one of the seeds. This observation transfers most of the structure obtained in the seedless case and reduces the remaining work to characterizing when a seed has infinite degree in its genealogical component.

\medskip 

Section \ref{s:local_limits} returns to finite trees and proves the local-limit results. We first record the local convergence of genealogical graphs under the natural coupling. We then treat the two-seed coloring, proving Theorem~\ref{theo:twoseed} and Proposition~\ref{prop:twoseed} by reading the limiting color of a vertex from the terminal seed or the asymptotic direction of its infinite-volume ancestral path. The last part of the section studies the coloring with a blue seed at the root and red seeds at the leaves: the local convergence and the statement \eqref{eq:min} follow from the infinite-volume results with seeds. The remaining estimate \eqref{eq:max}, which concerns the behavior of the coloring near the leaves, requires different arguments; it is proved separately in Section \ref{sec:max}, where we follow the coloring dynamics forward in time and compare the recursive growth of blue vertices with a supercritical Galton--Watson process.

\newpage

\section{The genealogical process}\label{s:genealogicalprocess}

\subsection{Construction of the genealogical process}\label{sec:constructionW}
A central object in the study of the geometry of the genealogical graph is the ancestral path $\Gamma(v)$ of a vertex $v$, as defined in \eqref{eq:def_path}. To analyze this process, it is convenient to keep track not only of $\Gamma(v)$ but also of the set of vertices of $\T$ that need to be examined in order to determine $\Gamma(v)$. We denote this set by $W_\infty(v)$. This subsection gives a precise construction of $\Gamma(v)$ and $W_\infty(v)$. We first introduce the notion of a \emph{uniform increasing indexing}, which gives us a rule to break ties when several possible ancestors of a vertex are available.

\begin{defi}\label{def:randomindex}
Let $G=(V,E)$ be a connected, locally finite graph, and let $v \in V$.
We say that a random sequence $(\zi_j)_{j \ge 1}$ is a uniform increasing indexing of the vertices of $G$ centered at $v$ if it satisfies the following properties:
\begin{itemize}
    \item[(i)] (indexing of $V$) $\{\zi_j : j \ge 1\} = V$ and $\zi_j \neq \zi_i$ for $j \neq i$;
    \item[(ii)] (increasing) if $j < i$, then $d(v,\zi_j) \le d(v,\zi_i)$;
    \item[(iii)] (uniformity) for every $r \ge 0$, the order of the vertices at distance $r$ from $v$ induced by $(\zi_j)_{j \ge 1}$ is uniform among all possible orderings and independent of the orderings induced at other distances.
\end{itemize}
\end{defi} 
Note that (iii) makes sense because we consider locally finite graphs, so the set of vertices at distance $r$ from $v$ is finite for every $r \ge 0$. In particular, a uniform increasing indexing of the vertices of $\T$ centered at $v$ can be constructed by first sampling a uniform random ordering of the vertices at distance $1$ from $v$, then a uniform random ordering of the vertices at distance $2$ from $v$, and so on, independently for each distance.

Recall that $(\tau_u)_{u \in \T}$ denotes the arrival times associated with the vertices of $\T$. In this section, we consider the model without seeds; hence, the variables $(\tau_u)_{u \in \T}$ are i.i.d.\ with an atomless distribution.
For each vertex $w \in \T$, let $(\zi_j^w)_{j \geq 1}$ denote a uniform increasing indexing of $\T$ centered at $w$, independent of the family $(\tau_u)_{u \in \T}$, and such that the families $\bigl((\zi_j^w)_{j \geq 1}\bigr)_{w \in \T}$ are independent.
Using this collection of random variables, we define the direct ancestor $p(v)$ of $v$ as follows: we examine the arrival times of the vertices of $\T$ in the order given by the sequence $(\zi_j^v)_{j \ge 1}$, and we stop as soon as we find a vertex $u$ such that $\tau_u < \tau_v$. We then declare that $p(v)$ is this vertex $u$. More formally, we set
\[
J(v):=\inf\{j\ge 1:\tau_{\zi_j^v}<\tau_v\},
\qquad
p(v):=\zi_{J(v)}^v.
\]
In this construction, the family of random variables $(\tau_u)_{u \in \T}$ uniquely determines the distance between $v$ and its direct ancestor $p(v)$, while the family of random indexings is used to break ties when several possible ancestors are available. Moreover, because the ordering is uniform on each sphere centered at $v$, this choice is uniform among all closest such vertices; therefore, this definition of the direct ancestor coincides with the definition given in the introduction.

For $v \in \T$, we can generalize this construction to define the entire ancestral path $\Gamma(v)=(V_n)_{n \ge 1}$ of $v$ by induction. We simultaneously construct a sequence of sets $(W_n(v))_{n \ge 1}$: the set $W_n(v)$ records the vertices whose arrival times have been revealed in order to determine $(V_1,\ldots,V_n)$.

\medskip

First, set
\[
V_1 := V_1(v) := v,
\qquad
W_1 := W_1(v) := \{v\}.
\]

Assume now that $(V_1,\dots,V_n)$ and $W_n$ have been constructed, where $V_n$ is the $(n-1)$-th ancestor of $v$, and $W_n$ records the vertices revealed up to this point. Assume moreover that, for every $u \in W_n$, we have $\tau_u \ge \tau_{V_n}$.
To determine the direct ancestor of $V_n$, we examine the vertices in the order given by the increasing indexing centered at $V_n$ and set
\[
J_n := \inf\{j \ge 1 : \tau_{\zi_j^{V_n}} < \tau_{V_n}\}
     = \inf\{j \ge 1 : \zi_j^{V_n} \notin W_n \hbox{ and } \tau_{\zi_j^{V_n}} < \tau_{V_n}\}.
\]
The two definitions coincide since, for every $u \in W_n$, we have $\tau_u \ge \tau_{V_n}$.
Thus, by the construction of the direct ancestor, $\zi_{J_n}^{V_n}$ is precisely equal to $p(V_n)$. We therefore set
\begin{equation}\label{eq:defvraiVn}
 V_{n+1} := \zi_{J_n}^{V_n},
\qquad
W_{n+1} := W_n \cup \{\zi_j^{V_n} : j \le J_n\}.
\end{equation}

By construction, $W_{n+1}$ records all vertices whose arrival times have been examined in order to determine $(V_1,\ldots,V_{n+1})$, and for every $u \in W_{n+1}$, we have $\tau_u \ge \tau_{V_{n+1}}$.

Finally, define
\[
\Gamma(v):=(V_n)_{n\ge 1},
\qquad
W_\infty := W_\infty(v) := \bigcup_{n \ge 1} W_n,
\]
so that $\Gamma(v)$ is the ancestral path of $v$, and $W_\infty$ is the set of vertices whose arrival times have been ultimately revealed. The process $(V_n(v), W_n(v))_{n \ge 1}$ will play a central role throughout the paper, and we therefore introduce the following terminology.

\begin{defi}
For \( v \in \T \), the process \( (V_n(v), W_n(v))_{n \ge 1} \) constructed by the procedure described above is called the \emph{ancestral process} of \( v \). Its first coordinate \( (V_n(v))_{n \ge 1} \) is called the \emph{ancestral path} of \( v \).
\end{defi}

For a fixed $v$, it will be convenient to consider a process $(\hat{V}_n(v), \hat{W}_n(v))_{n \ge 1}$ with the same distribution as the ancestral process of $v$, but defined directly in terms of 
the inter-arrival times $(I_n)_{n \ge 1}$ of a generic lower record process associated with an i.i.d.\ atomless sequence\footnote{One concrete realization of $(I_n,\, n\geq 1)$ is the following. Let $(X_i)_{i \ge 1}$ be i.i.d.\ with an atomless distribution, and set $T_0=0$, $T_1=1$, and, for $n\ge1$, $T_{n+1}:=\inf\{k>T_n:X_k<X_{T_n}\}$. Then set $I_n:=T_n-T_{n-1}$ for $n\ge 1$. The law of $(I_n, n\ge 1)$ does not depend on the choice of the distribution of $(X_i)_{i \ge 1}$.}. For each $w \in \T$, let $(\zi_j^w)_{j \geq 1}$ be a uniform increasing indexing of $\T$ centered at $w$, independent of $(I_n)_{n \geq 1}$, and such that the families $\bigl((\zi_j^w)_{j \geq 1}\bigr)_{w \in \T}$ are independent. Set 
\[
\hat{V}_1 := \hat{V}_1(v) := v,
\qquad
\hat{W}_1 := \hat{W}_1(v) := \{v\}.
\]
By induction, assuming that $(\hat{V}_1,\dots,\hat{V}_n)$ and $\hat{W}_n$ have been constructed, define $\hat{J}_n$ as the first index such that we explore $I_{n+1}$ new vertices around $\hat{V}_n$ that do not belong to $\hat{W}_n$:
\begin{equation}\label{eq:def1process}
\hat{J}_n := \inf\left\{j \ge 1 : \sharp\{i \le j : \zi_i^{\hat{V}_n} \notin \hat{W}_n\}=I_{n+1}\right\}.
\end{equation}
Then set
\begin{equation}\label{eq:def2process}
\hat{V}_{n+1} := \zi_{\hat{J}_n}^{\hat{V}_n},
\qquad
\hat{W}_{n+1} := \hat{W}_n \cup \{\zi_j^{\hat{V}_n} : j \le \hat{J}_n\}.
\end{equation}
We note that this construction uses only the inter-arrival times $(I_n)_{n \ge 1}$ of a generic lower record process, rather than the arrival times themselves. Nevertheless, as the next lemma states, it has the same distribution as the ancestral process.

\begin{lemma}\label{lemm:const_record} The process $(\hat{V}_n(v), \hat{W}_n(v))_{n \ge 1}$ defined above has the same distribution as the ancestral process $(V_n(v), W_n(v))_{n \ge 1}$. Moreover, for every $n\ge 1$, we have
\[
I_n = \sharp\bigl(\hat{W}_n(v) \setminus \hat{W}_{n-1}(v)\bigr),
\qquad \text{with the convention } \hat{W}_0(v) := \emptyset.
\]
\end{lemma}
\begin{proof}
The result follows directly by induction from the construction. Conditionally on the past up to step $n$, the unexplored arrival times, read in the random indexing around $V_n$, form a fresh i.i.d.\ atomless sequence. Since $\tau_{V_n}$ is the current lower record among the revealed times, the number of new vertices inspected before finding the next ancestor has the law of $I_{n+1}$, and the last inspected vertex is chosen through the same uniform increasing indexing. This is exactly the recursive construction of $(\hat{V}_n,\hat{W}_n)$. The identity for $I_n$ is also immediate, with $I_1=1$ corresponding to $\hat{W}_1=\{v\}$.
\end{proof}

In the rest of the paper, whenever we study a single ancestral process, we will usually favor the ``record-based'' construction above and omit the hats, so we will simply write $(V_n(v), W_n(v))_{n \ge 1}$ for $(\hat{V}_n(v), \hat{W}_n(v))_{n \ge 1}$.

\subsection{Preliminary results on the genealogical process}\label{sec:preliminaires}

We now present some preliminary results for the ancestral process. We use the construction of the previous section from the inter-arrival times $(I_k)_{k \ge 1}$ and the random indexings $\bigl((\zi_j^w)_{j \geq 1}\bigr)_{w \in \T}$ through Equations~\eqref{eq:def1process} and~\eqref{eq:def2process}. Since $\T$ is transitive, it will be enough to study the process starting from $v = o$ and, for the sake of readability, we will simply write $(V_n, W_n)_{n \ge 1}$ for $(V_n(o), W_n(o))_{n \ge 1}$.

\medskip 

For $n\ge 2$, define $R_n$ as the graph distance between $V_n$ and $V_{n-1}$:
\begin{equation}\label{eq:defR_n}
    R_n:=d(V_n,V_{n-1}).
\end{equation}
Note that $R_n$ is a function of $V_{n-1},W_{n-1},I_n$ given by
\begin{equation}\label{eq:defR_n2}
R_n = \inf \{ r\ge 0:\; \sharp \{ u\in \T\setminus W_{n-1}, d(V_{n-1},u)\le r \} \geq I_n \}.
\end{equation}
Using this equation, we can give a precise estimate of $R_n$ in terms of $I_n$. Recall that $\T_u$ denotes the descendant subtree rooted at $u$ (for $u\neq o$, it is the connected component of $\T\setminus \{u\}$ that does not contain the root $o$).

\begin{lemma}\label{lem:lienIetR}
    For any $n\ge 1$, we have the following properties:
    \begin{itemize}
        \item $W_n$ is a connected subset of $\T$. 
        \item $\T_{V_n}\cap W_n=\{V_n\}$. 
    \end{itemize} Moreover, for any $n\ge 2$, we have
    \begin{equation}\label{eq:In}
        \frac{I_n}{K+1}\le K^{R_n}\le K I_n.
    \end{equation}
\end{lemma}

\begin{proof} Let us first note that, by construction, $V_{n}\in W_{n}$ for all $n\ge 1$. 
    Let us now prove by induction on $n$ that $W_n$ is connected and $\T_{V_n}\cap W_n=\{V_n\}$. This holds for $n=1$. Assume it holds for $n-1$. Let $(\zi_j^{n-1})_{j\ge 1}$ be the subsequence of $(\zi_j^{V_{n-1}})_{j\ge 1}$ obtained by keeping only the terms that do not belong to $W_{n-1}$, and set
    $$Z_l:=W_{n-1}\cup \{\zi_j^{n-1}: j\le l\},$$
    so that $Z_0=W_{n-1}$ and $Z_{I_n}=W_n$. Since $V_{n-1}\in W_{n-1}$ and $\zi_j^{V_{n-1}}$ is an increasing indexing centered at $V_{n-1}$, we see by induction on $l$ that $Z_l$ is a connected subset for all $0\le l\le I_n$ and so $W_n$ is connected. Moreover, $W_n=Z_{I_n-1}\cup\{V_n\}$ and $Z_{I_n-1}$ is a connected subset which does not contain $V_n$. Hence, it is contained in one of the connected components of $\T\setminus \{V_n\}$. Since the root $o$ is in $Z_{I_n-1}$, the connected component of $\T\setminus \{V_n\}$ containing $Z_{I_n-1}$ is necessarily the one containing the root. Hence, we get 
    $$Z_{I_n-1}\cap \T_{V_n}=\emptyset,$$ 
    and so $$\T_{V_n}\cap W_n=\{V_n\}.$$

    Let us now prove \eqref{eq:In}. On the one hand, using \eqref{eq:defR_n2}, we have 

\begin{eqnarray*}
    I_n &\le & \sharp\{u\in \T\setminus W_{n-1}, d(V_{n-1},u)\le R_{n}\}\\
    &\le& \sharp\{u\in \T\setminus \{V_{n-1}\}, d(V_{n-1},u)\le R_{n}\}\\
    &= & \sum_{i=0}^{R_n-1}  (K+1)\cdot K^{i}\le (K+1)K^{R_n}.
\end{eqnarray*}
This gives the lower bound. The upper bound is trivial when $R_n = 1$. We therefore assume $R_n \geq 2$, and bound $I_n$ from below using again \eqref{eq:defR_n2} combined with the fact that $\T_{V_{n-1}}\cap W_{n-1}=\{V_{n-1}\}$ to conclude that
\begin{eqnarray*}
    I_n &> & \sharp\{u\in \T\setminus W_{n-1}, \; d(V_{n-1},u)< R_{n}\}\\
    &\ge& \sharp\{u\in \T_{V_{n-1}}\setminus\{V_{n-1}\}, \; d(V_{n-1},u)< R_{n}\}\\
    &\ge& K^{R_n-1}.
\end{eqnarray*}
\end{proof}

The process $(I_n)_{n\ge 1}$ is the inter-arrival-time process of a classical record process. Such processes have been well studied. We will use the following result, which is an easy consequence of the classical asymptotic behavior of record processes.

\begin{prop}[see Arnold \cite{arnold2011}]\label{prop:Arnold} Let $(I_n)_{n\ge 1}$ be the inter-arrival times of a record process of an i.i.d.\ sequence with diffuse distribution. Then
\begin{enumerate}
    \item[(i)] We have $$\lim_{n\to \infty } \frac{\log I_n}{n}=1 \mbox{ a.s.}.$$
    \item[(ii)] For any $n\ge 1$ and $i_2,\ldots,i_n\ge 1$,
$$\P(I_2=i_2,\ldots, I_n=i_n,\ 2^{k}\le I_k\le 3^{k} \hbox{ for all } k> n)>0.$$
    \item[(iii)] For any $c\ge 0$ and $n\ge 1$, we have $$\P(\forall m\ge 0, I_{n+m+1}\ge c \mid I_1,\ldots,I_n)\ge \P(\forall m\ge 0, I_{m+2}\ge c)>0 \quad\hbox{a.s.}.$$
\end{enumerate}
\end{prop}

\begin{proof}
The first part of the proposition is a classical result (see, for example, \cite{arnold2011}, p.~28). For the second point, one uses the Markov description of the lower record process: any finite admissible sequence of inter-arrival times has positive probability, and after such a finite prescription the future process is again a lower record process started from the current record value. Since \(2<e<3\), point~(i) implies that, for some sufficiently large deterministic \(N\), the tail event \(\{2^k\le I_k\le 3^k \text{ for all } k>N\}\) has positive probability; the finitely many constraints for \(n<k\le N\) can then be imposed with positive probability.

Finally, the last point follows from the standard Markov description of the record-time chain: conditionally on \(I_1,\ldots,I_n\), the future gaps are stochastically lower bounded by a fresh record inter-arrival process started from time \(1\). Point~(ii) then implies that
$\P(\forall m \ge 0,\ I_{m+2} \ge c) > 0$.
\end{proof}

We now introduce the filtration $(\mathcal{F}_n)_{n\ge 1}$ given by 
\begin{equation}\label{eq:defFn}
    \mathcal{F}_n:=\sigma\bigl((I_k)_{k\ge 1}, ((\zi_j^{V_i})_{j\ge 1})_{i< n}\bigr).
\end{equation}
Notice that $\mathcal{F}_n$ contains the whole sequence of inter-arrival times $(I_k)_{k\ge 1}$, not only the first $n$ of them.

\begin{lemma} 
For any $n\ge 1$, $V_n$, $W_n$ and $R_{n+1}$ are $\mathcal{F}_n$-measurable.  
Moreover, conditionally on $\mathcal{F}_n$, the vertex $V_{n+1}$ is uniformly distributed on the set of vertices $u\in \T\setminus W_n$ such that $d(V_n,u)=R_{n+1}$. 
\end{lemma}

\begin{proof}
The measurability assertion follows by induction. Indeed, $V_1$ and $W_1$ are deterministic, and if $V_n$ and $W_n$ are $\mathcal F_n$-measurable, then $R_{n+1}$ is $\mathcal F_n$-measurable by \eqref{eq:defR_n2}, since all the variables $(I_k)_{k\ge1}$ are included in $\mathcal F_n$. Once the indexing centered at $V_n$ is added, the definitions \eqref{eq:def1process} and \eqref{eq:def2process} give $V_{n+1}$ and $W_{n+1}$. The conditional distribution of $V_{n+1}$ follows from the fact that $(\zi_j^{V_n})_{j\ge 1}$ is independent of $\mathcal{F}_n$ and is chosen uniformly among the increasing indexings centered at $V_n$.
\end{proof}

\begin{prop}\label{prop:V_nconditionnelle}
    For any $\varepsilon\in (0,1)$ and $n\ge 1$,
    $$\P(|V_{n+1}|< |V_{n}|+R_{n+1}(1-\varepsilon) \mid \mathcal{F}_{n})\le 2K^{-\frac{\varepsilon R_{n+1}}{2}}.$$
\end{prop}

\begin{proof} First note that the proposition is trivial for $n=1$ since $V_1=o$ and so $|V_2|=R_2$. Assume now that $n>1$. Conditionally on $\mathcal{F}_n$, $V_{n+1}$ is uniform on the set of vertices $u\in \T\setminus W_n$ such that $d(V_n,u)=R_{n+1}$. Hence,
\begin{equation}\label{eq:probaV}
    \P(|V_{n+1}|< |V_{n}|+R_{n+1}(1-\varepsilon) \mid \mathcal{F}_{n})=\frac{\sharp \{ u\in \T\setminus W_n,\; d(V_n,u)=R_{n+1}, \; |u|<|V_n|+R_{n+1}(1-\varepsilon)\}}{\sharp \{ u\in \T\setminus W_n,\; d(V_n,u)=R_{n+1}\}}.
\end{equation}
    First, recalling that $\T_{V_n} \cap W_n=\{V_n\}$, we get 
    $$\sharp \{ u\in \T\setminus W_n,\; d(V_n,u)=R_{n+1}\}\ge \sharp \{ u\in \T_{V_n}, \;d(V_n,u)=R_{n+1}\}= K^{R_{n+1}}.$$
    Now let $u$ be such that $d(u,V_n)=R_{n+1}$ and $|u|<|V_n|+R_{n+1}(1-\varepsilon)$.
    Given two vertices $u,w$, we denote by $w\wedge u$ their common ancestor of maximal height. Using that $2|w\wedge u|=|u|+|w|-d(u,w)$, we get 
    $$2|V_n\wedge u|= |V_n|+ |u|-d(u,V_n)< 2|V_n|-\varepsilon R_{n+1},$$
     and so 
    $$|V_n\wedge u|< |V_n|-\frac{\varepsilon R_{n+1}}{2}.$$
    If $|V_n|<\lceil\varepsilon R_{n+1}/2\rceil$, this is impossible and the set under consideration is empty. Otherwise, let $w$ be the vertex on the path from the root to $V_n$ such that $|w|=|V_n|-\lceil\varepsilon R_{n+1}/2\rceil$. We get $d(w,u)=d(u,V_n)-d(V_n,w)= R_{n+1}- \lceil\frac{\varepsilon R_{n+1}}{2}\rceil$. Thus,
    \begin{align*}
    &\sharp \{ u\in \T\setminus W_n:\; d(V_n,u)=R_{n+1}, \; |u|<|V_n|+R_{n+1}(1-\varepsilon)\}\\
    &\qquad\le \sharp \left\{ u\in \T:\; d(w,u)=R_{n+1}- \left\lceil{\frac{\varepsilon R_{n+1}}{2}}\right\rceil\right\}
    \le 2K^{(1-\varepsilon/2)R_{n+1}}.
    \end{align*}
    Plugging this in \eqref{eq:probaV}, we conclude that
    \begin{equation*}
    \P(|V_{n+1}|< |V_{n}|+R_{n+1}(1-\varepsilon) \mid \mathcal{F}_{n})\le\frac{2K^{(1-\varepsilon/2)R_{n+1}}}{K^{R_{n+1}}}=2K^{-\frac{\varepsilon R_{n+1}}{2}}.
\end{equation*}
\end{proof}


\begin{corollary}\label{coro:V_nenn2} The following limits hold a.s.:

$$\lim_{n\to \infty} \frac{ R_{n}}{n}=  \frac{1}{\log K}\qquad \hbox{ and }\qquad  \lim_{n\to \infty} \frac{ |V_{n}|}{n^2}=  \frac{1}{2\log K}.$$
\end{corollary}

\begin{proof} The limit of $R_n/n$ is a direct consequence of Lemma \ref{lem:lienIetR} and Proposition \ref{prop:Arnold}.
Moreover, for any fixed $\varepsilon\in(0,1)$, this linear growth of $R_n$ makes the upper bounds in Proposition~\ref{prop:V_nconditionnelle} summable. Hence, by the conditional Borel-Cantelli lemma, there exists a.s. $N$ such that, for $n\ge N$,
\begin{equation}\label{eq:minVn}
    |V_{n+1}| \ge |V_{n}|+ R_{n+1}(1-\varepsilon).
\end{equation}
Thus, for all sufficiently large $n$,
$$ R_{n+1}(1-\varepsilon) \le |V_{n+1}| - |V_{n}| \le  R_{n+1}.$$
Using the asymptotic
$$\sum_{k=2}^n R_k \sim \frac{n^2}{2\log K},$$
and letting $\varepsilon$ tend to 0 yields 
$$\lim_{n\to \infty} \frac{ |V_{n}|}{n^2}=  \frac{1}{2\log K}\qquad a.s.$$
\end{proof}

\section{Study of $\mathcal{G}$ without seeds}\label{s:Gwithoutseed}
The aim of this section is to prove Theorem~\ref{thm:GwithoutSeeds}. We do not follow the order of the statements. Instead, the proof is organized around the geometry of ancestral paths: we first show that each component has an asymptotic direction, then use local constructions to produce many components, and finally prove that distinct components have distinct directions and that descendant sets are finite. More precisely, we proceed as follows:
\begin{itemize}
    \item Item 3: we prove that each connected component of $\mathcal{G}$ admits a limiting direction;
    \item Items 1 and 4: we prove that $\mathcal{G}$ has infinitely many connected components, that for every subtree $\T_v$ there exists a component whose limiting direction belongs to $\partial\T_v$, and we derive a lower bound on the number of components intersecting the sphere of radius $n$;
    \item Item 5: we show that two distinct connected components almost surely have distinct limiting directions;    
    \item Item 2: we prove that every vertex almost surely has finitely many descendants.
\end{itemize}
\subsection{Proof of Theorem \ref{thm:GwithoutSeeds}, Item 3}
We begin by proving that each connected component of $\mathcal{G}$ has a limiting direction. Since $\T$ is vertex-transitive, it is sufficient to prove this property for the connected component containing the root. 
\begin{lemma} Let $\Gamma(o)=(V_1,V_2,\ldots)$ be the ancestral path of the root. Almost surely, for all sufficiently large $n$,
\[
 |V_{n+1} \wedge V_n| > |V_n \wedge V_{n-1}|,
\]
where $V_n\wedge V_{n-1}$ denotes the common ancestor of $V_n$ and $V_{n-1}$ of maximal height.
\end{lemma}

\begin{proof}
Recall that the ancestral path of the root $\Gamma(o) = (V_1,V_2,\ldots)$ is defined in Section~\ref{s:genealogicalprocess}. Recall also the notation \( R_n := d(V_n, V_{n-1}) \). Note that 
$$2|V_n\wedge V_{n+1}|=|V_n|+|V_{n+1}|-d(V_n,V_{n+1}).$$
Corollary \ref{coro:V_nenn2} states that $R_n/n$ tends a.s. to $1/\log K$, so in particular, for $n$ large enough, we have
$$R_n> \frac{R_{n+1}}{2}.$$
Moreover, we have proved (see \eqref{eq:minVn}) that, for $n$ large enough, it holds that
$$|V_{n+1}|\ge |V_{n}|+\frac{R_{n+1}}{2}.$$
Thus, a.s., for $n$ large enough, we have 
$$2|V_n\wedge V_{n+1}|=|V_n|+|V_{n+1}|-R_{n+1}\ge 2|V_n| -\frac{R_{n+1}}{2} \ge  2|V_{n-1}|+R_{n} -\frac{R_{n+1}}{2}> 2|V_{n-1}|.$$
Hence
$$|V_{n}\wedge V_{n+1}|> |V_{n}\wedge V_{n-1}|.$$
\end{proof}

\begin{proof}[Proof of Theorem \ref{thm:GwithoutSeeds}, Item 3]
Any connected component of $\mathcal{G}$ is an oriented tree and $u,v\in \T$ are in the same connected component if and only if their ancestral paths coalesce. So, to prove Item 3 of Theorem \ref{thm:GwithoutSeeds}, it is sufficient to find an infinite ray $z=(z_i)_{i\ge 1}\in \partial \T$ such that, for every $i \ge 1$, the vertices of $\Gamma(o)$ eventually belong to $\T_{z_i}$.

\begin{figure}
    \centering
\includegraphics[width=11cm]{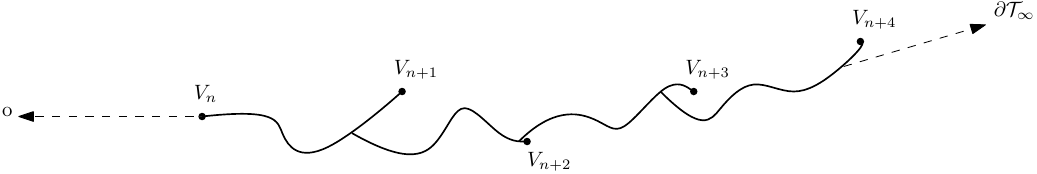}
\caption{\small{Illustration of the geometry of an ancestral path for large $n$.}}\label{fig:path_image}
\end{figure}

Note that $|V_{n}\wedge V_{n+1}|> |V_{n}\wedge V_{n-1}|$ implies that 
$V_{n}\wedge V_{n+1}\in \T_{V_{n}\wedge V_{n-1}}$ (see Figure \ref{fig:path_image} for an illustration of a path satisfying this condition). Thus, by the previous lemma, there exists an
infinite ray $z=(z_i)_{i\ge 1}$ that contains the vertices $V_{n}\wedge V_{n+1}$ for large $n$. By construction, for every $i \ge 1$, the vertices of $\Gamma(o)$ eventually belong to $\T_{z_i}$ and so the ancestral path of the root has a limiting direction given by $z$.

\end{proof}

\subsection{Proof of Theorem \ref{thm:GwithoutSeeds}, Items 1 and 4}

We prove here that $\mathcal{G}$ has infinitely many connected components and, in order  
to control the number of components in each direction of the space, we prove that, 
with positive probability, a given neighbor of the root is never explored by the ancestral process of the root.

\begin{prop}\label{prop:keyresult} Let $v\in \T$ be a neighbor of the root. Let $(V_n,W_n)$ be the ancestral process of the root and $W_\infty:=\cup_{n\ge 1} W_n$. Then
$$\P(v\notin W_\infty) = \P(\T_v \cap W_\infty = \emptyset) > 0.$$ 
\end{prop}
This proposition shows that, with positive probability, it is not necessary to reveal the arrival times of the vertices in $\T_v$ in order to determine the ancestral line of the root. In particular, on this event, the ancestral path of $o$ is unaffected by the presence of a seed anywhere in the subtree $\T_v$.

\medskip 
We use the following lemma.

\begin{lemma}\label{lemm:AcapB} Let $v\in \T$ be a neighbor of the root and let $n_0\ge 7$. Let $w_1:=o$, let $w_2$ be a neighbor of the root distinct from $v$, and, for $3\le i\le n_0$, let $w_i$ be a vertex of $\T_{w_{i-1}}$ such that $|w_i|=2|w_{i-1}|$. Set 
    $$
    \mathcal{A}:=\{ V_i=w_i \hbox{ and } v \notin W_i\hbox{ for } 1\leq i\le n_0\}.
    $$
    $$\mathcal{B}:=\Big\{\forall n> n_0,\, \frac{n}{2}\le \log(K)R_n \le 2n \mbox{ and } |V_n|\ge |V_{n-1}|+\frac{R_n}{2}\Big\}.$$
Then, we have $$\mathcal{A}\cap \mathcal{B} \subset \{v\notin W_\infty\}.$$
\end{lemma}
\begin{proof} 
The event $\mathcal{A}$ ensures that $W_{n_0}$ does not contain $v$, that $V_{n_0}\in\T_{w_2}$, and that $|V_{n_0}|=2^{n_0-2}\ge 32$. Assume now that $\mathcal{B}$ also holds and let us prove by induction that $v\notin W_n$ and $V_n\in\T_{w_2}$ for all $n\ge n_0$.
For $i\ge n_0+1$, the event $\mathcal{B}$ gives
$$|V_{i}|\ge |V_{i-1}|+\frac{i}{4 \log K},$$
which yields, for $n\ge n_0$,
$$|V_{n}|\ge  |V_{n_0}|+\sum_{i=n_0+1}^n \frac{i}{4 \log K}\ge 2^{n_0-2}+\big(n(n+1)-n_0(n_0+1)\big)\frac{1}{8\log K}.$$
Thus, using that $R_{n+1}\log K\le 2(n+1)$, we get  
$$|V_n|-R_{n+1}\ge 2^{n_0-2}+\big(n(n+1)-n_0(n_0+1)-16(n+1)\big)\frac{1}{8\log K}.$$
For $n\ge n_0\ge 7$, the last expression is bounded below by
$$2^{n_0-2}-\frac{16(n_0+1)}{8\log K}\ge 2^{n_0-2}-\frac{2(n_0+1)}{\log 2}>0.$$
Since $R_{n+1}=d(V_n,V_{n+1})$ and $W_{n+1}\setminus W_n \subset B(V_n,R_{n+1})$, the inequality $R_{n+1}<|V_n|$ and the induction hypothesis $V_n\in\T_{w_2}$ imply that $B(V_n,R_{n+1})\subset\T_{w_2}$. This implies in particular that $v\notin W_{n+1}\setminus W_n$ and $V_{n+1}\in\T_{w_2}$, and so, by induction, that $v\notin W_\infty$.
\end{proof}

\begin{proof}[Proof of Proposition \ref{prop:keyresult}] Choose $n_0\ge 7$ large enough so that, for every $n>n_0$,
$$(K+1)e^{n/2}\le 2^n\le 3^n\le \frac{1}{K}e^{2n}.$$
It remains to prove that $\P(\mathcal{A}\cap \mathcal{B})>0$, where $\mathcal A$ and $\mathcal B$ are defined in Lemma~\ref{lemm:AcapB} with this choice of $n_0$. Recall the definition of the filtration $\mathcal{F}_n:=\sigma\bigl((I_k)_{k\ge 1}, ((\zi_j^{V_i})_{j\ge 1})_{i< n}\bigr)$ given in \eqref{eq:defFn} and note that $\mathcal{A}\in \mathcal{F}_{n_0}$ and has positive probability. Set $\mathcal{B}_{n_0}:=\Omega$ and, for $N\ge n_0+1$, define
$$\mathcal{B}_N:=\Big\{\forall \; n_0< n\le N,\; \frac{n}{2}\le \log(K)R_n \le 2n \mbox{ and } |V_n|\ge |V_{n-1}|+\frac{R_n}{2}\Big\}.$$
By monotone convergence, $\P(\mathcal{B}\mid \mathcal{F}_{n_0})=\lim_{N\to \infty} \P(\mathcal{B}_N\mid \mathcal{F}_{n_0})$. Moreover, by the choice of $n_0$ and using Lemma \ref{lem:lienIetR}, we see that 
$$2^n\le I_n\le 3^n \;\implies\; (K+1) e^{\frac{n}{2}}\le I_n \le \frac{1}{K}e^{2n} \;\implies\; e^{\frac{n}{2}}\le K^{R_n} \le e^{2n} \;\implies\; \frac{n}{2}\le \log(K)R_n \le 2n.$$
Then, for $N\ge n_0+1$, using that $\mathcal{B}_{N-1}$ and $R_N$ are $\mathcal{F}_{N-1}$-measurable, we write
\begin{eqnarray*}
    \P(\mathcal{B}_N\mid \mathcal{F}_{n_0})&=&\E(1_{\mathcal{B}_{N-1}}1_{\{\frac{N}{2}\le \log(K)R_N \le 2N\}}\E(1_{\{|V_N|\ge |V_{N-1}|+\frac{R_N}{2}\}}\mid\mathcal{F}_{N-1}) \mid \mathcal{F}_{n_0}) \\
    &\ge & \E(1_{\mathcal{B}_{N-1}}1_{\{\frac{N}{2}\le \log(K)R_N \le 2N\}}(1-2K^{-\frac{R_{N}}{4}}) \mid \mathcal{F}_{n_0})\\
     &\ge & \E(1_{\mathcal{B}_{N-1}}1_{\{\frac{N}{2}\le \log(K)R_N \le 2N\}}(1-2K^{-\frac{N}{8\log K}}) \mid \mathcal{F}_{n_0})\\
     &\ge & (1-2K^{-\frac{N}{8\log K}})\E(1_{\mathcal{B}_{N-1}}1_{\{2^N\le I_N \le 3^N\}} \mid \mathcal{F}_{n_0})\\
  &=& (1-2K^{-\frac{N}{8\log K}}) 1_{\{2^N\le I_N \le 3^N\}}\P(\mathcal{B}_{N-1}\mid \mathcal{F}_{n_0}).
\end{eqnarray*}
Hence, we get 
$$\P(\mathcal{B}\mid \mathcal{F}_{n_0})\ge \left(\prod_{n>n_0}(1-2K^{-\frac{n}{8\log K}}) \right)1_{\{\forall n>n_0,\ 2^n\le I_n \le 3^n\} }:=c1_{\{\forall n>n_0,\ 2^n\le I_n \le 3^n\} },$$
for some $c>0$. Finally, we write 
 $$\P(\mathcal{A}\cap \mathcal{B})=\E(1_{\mathcal{A}} \P(\mathcal{B}\mid \mathcal{F}_{n_0}))\ge c \,\E(1_{\mathcal{A}}1_{\{\forall n>n_0,\ 2^n\le I_n \le 3^n\} }).$$
Since the event $\mathcal{A}$ only depends on the value of $I_k$ for $k\le n_0$ and on the indexing $(\zi_j^{V_k})_{j\ge 1}$ for $k<n_0$, using Proposition \ref{prop:Arnold}, we conclude that 
$\P(\mathcal{A}\cap \mathcal{B})>0$, which, in view of Lemma \ref{lemm:AcapB}, implies that $\P(v\notin W_\infty)>0$.
\end{proof}

\begin{proof}[Proof of Theorem \ref{thm:GwithoutSeeds}, Items 1 and 4]
We now prove that $\mathcal{G}$ has infinitely many connected components. Since we need to consider the ancestral paths of several vertices at the same time, we use the construction of $W_n(v)$ given in \eqref{eq:defvraiVn} in Section \ref{sec:constructionW} (and not the one given by \eqref{eq:def2process}). Since $\T$ is transitive, Proposition \ref{prop:keyresult} implies that, for $v \in \T \setminus \{o\}$, if $\overleftarrow{v}$ denotes the father of $v$ for the tree structure $\T$, then
$$c := \P(\overleftarrow{v}\notin W_\infty(v)) >0.$$
Let us define 
$$n_0(v) := \sup\{n\ge 1: \overleftarrow{v}\notin W_n(v)\} \in \N \cup \{\infty\}.$$
We define $W'(v)$ as the set of vertices explored until we explore $\overleftarrow{v}$, that is
\[
W'(v) =
\begin{cases}
W_{n_0(v)}(v)\cup\{\zi_{j}^{V_{n_0(v)}}:\; j\le j_0\} & \text{if } n_0(v)<\infty,\\
W_\infty(v) & \text{if } n_0(v)=\infty.
\end{cases},
\]
where $j_0$ is the unique index where $\zi_{j_0}^{V_{n_0(v)}} = \overleftarrow{v}$.

Let us note that the random sets $W'(v)$ for $v\in S(o,r)$ are independent. Indeed, by construction, $W'(v)$ only depends on the arrival times $\tau_w$ and the random indexing $(\zi_j^w)_{j\ge 1}$ for $w \in \T_v$ and we have $\T_{v}\cap\T_{v'}=\emptyset$ if $v,v'$ are two distinct vertices of $S(o,r)$. Moreover, we have

$$\{\hbox{The set } W'(v) \hbox{ is infinite}\}=\{n_0(v)=\infty\}=\{\overleftarrow{v}\notin W_\infty(v)\}\subset \{\Gamma(v) \mbox{ remains in } \T_v\}.$$ 
In particular, the number $N_r$ of connected components of $\mathcal G$ intersecting the sphere $S(o,r)$ is lower bounded by the number of vertices $v\in S(o,r)$ such that $n_0(v)=\infty$. Since these events are i.i.d.\ for $v\in S(o,r)$ and have positive probability $c$, we get 
$$\liminf_{r\to \infty} \frac{N_r}{\sharp S(o,r)} \geq c \mbox{ a.s.}.$$
This proves Item 1 of Theorem \ref{thm:GwithoutSeeds}. Furthermore, the same argument shows that, for every $v\in \T \setminus \{o\}$, there exist infinitely many vertices $u \in \T_v$ such that $n_0(u)=\infty$. In particular, there exists a $u\in \T_v$ such that $\Gamma(u)$ remains in $\T_u\subset \T_v$ and therefore the asymptotic direction of the genealogical component containing $u$ is in $\partial \T_v$. This proves Item 4 of Theorem \ref{thm:GwithoutSeeds}. 
\end{proof}

\subsection{Proof of Theorem \ref{thm:GwithoutSeeds}, Item 5}

In this section, we show that two distinct connected components cannot share the same asymptotic direction. In particular, for any $u,v\in \T$, either $\Gamma(u)$ and $\Gamma(v)$ ultimately coalesce or their asymptotic distance tends to infinity.
We first prove the following lemma concerning the ancestral process of the root.

\begin{lemma} \label{lemm:VncondHn} Let $(V_n, W_n)_{n\ge 1}$ be the ancestral process of the root. Let $\mathcal{H}_n=\sigma(V_k,W_k, k\le n)$ be the natural filtration of the exploration. There exists a constant $c>0$ such that, for any $n\ge 1$,
$$\P(\forall m\ge 0, V_{n+m}\in \T_{V_n} \mid \mathcal{H}_n)\ge c \mbox{ a.s.}.$$
\end{lemma}

Note that this result implies that for infinitely many $n$, $(V_{n+m})_{m\ge 0}$ remains in $\T_{V_n}$. In particular, we recover that $\dir \Gamma(o)$ exists a.s. (Item 3 of Theorem \ref{thm:GwithoutSeeds}).

\begin{proof}
For $u\in \T$, define the function $h_u$ on $\T$ by 
$$h_u(w):=\begin{cases}
    d(u,w) &\mbox{if } w\in \T_u\\
    -1 & \mbox{otherwise}
\end{cases}.$$
Hence, for $w\in \T_u$, $h_u(w)$ is the height of $w$ in $\T_u$. As before, we construct the process $(V_n, W_n)_{n\ge 1}$ from the inter-arrival times of a record process $(I_k)_{k\ge 1}$ and random indexings $((\zi_j^{v})_{j\ge 1},v\in \T)$. Recall that
$(\mathcal{F}_n)_{n\ge 1}$ denotes the filtration given by 
\begin{equation*}
    \mathcal{F}_n:=\sigma\bigl((I_k)_{k\ge 1}, ((\zi_j^{V_i})_{j\ge 1})_{i< n}\bigr).
\end{equation*}
Recall also that the distance $R_{k+1}=d(V_{k+1},V_{k})$ and $(V_k, W_{k})$ are $\mathcal{F}_{k}$-measurable, and that, knowing $\mathcal{F}_{k}$, $V_{k+1}$ is uniform on $\{u\in \T\setminus W_{k}:\ d(V_{k},u)=R_{k+1}\}$. Since $W_{k}\cap \T_{V_{k}}=\{V_{k}\}$, we also have
$$
\frac{\sharp \big(S(V_{k},R_{k+1})\cap \T_{V_{k}} \setminus W_k\big)}{\sharp \big(S(V_{k},R_{k+1}) \setminus W_k\big)} \;\geq \;
\frac{\sharp \big(S(V_{k},R_{k+1})\cap \T_{V_{k}}\big)}{\sharp\bigl(S(V_{k},R_{k+1})\bigr)} \;=\; \frac{K}{K+1}.$$
Thus, for $n\geq 1$ and $m \geq 0$, we find that
\begin{eqnarray*}
  \P(h_{V_n}(V_{n+m+1})=h_{V_n}(V_{n+m})+R_{n+m+1} \mid \mathcal{F}_{n+m}) & \geq &
  \P\big(V_{n+m+1} \in \T_{V_{n+m}} \mid \mathcal{F}_{n+m}\big)1_{\{V_{n+m}\in \T_{V_n}\}} \\
  & \geq & \frac{K}{K+1}1_{\{V_{n+m}\in \T_{V_n}\}}.  
\end{eqnarray*}
Let us now show that for $i\ge 1$, on the event $\{V_{n+m}\in \T_{V_n}\}$, we have 
$$\P(h_{V_n}(V_{n+m+1})\le h_{V_n}(V_{n+m})-i \mid \mathcal{F}_{n+m})\le K^{-i}.$$
\begin{figure}
    \centering
\includegraphics[width=5cm]{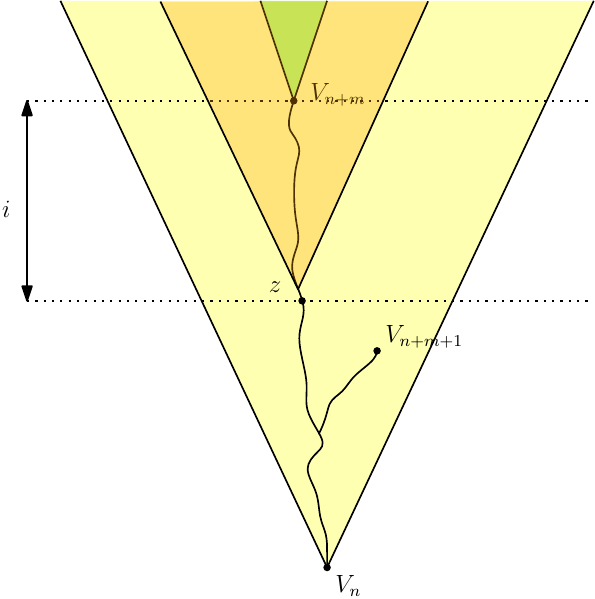}
\caption{\small{Illustration of the notation used in the proof of Lemma \ref{lemm:VncondHn}. In yellow: the tree $\T_{V_n}$. In orange: the tree $\T_{z\rightarrow V_{n+m}}$.}}\label{fig:item5}
\end{figure}
 
We can assume that $h_{V_n}(V_{n+m})-i\ge -1$ since the result is trivially true for larger $i$. If $i\le h_{V_n}(V_{n+m})=d(V_{n+m}, V_n)$, let $z$ be the unique vertex on the path $[V_n,V_{n+m}]$ such that $d(V_{n+m},z)=i$. If $i=h_{V_n}(V_{n+m})+1$, let $z$ be the father of $V_n$ for the tree structure (see Figure \ref{fig:item5} for an illustration of the notation). On the event $h_{V_n}(V_{n+m+1})\le h_{V_n}(V_{n+m})-i$, observe that $z\in [V_{n+m+1},V_{n+m}]$ and so 
$$V_{n+m+1}\in S(z, R_{n+m+1}-i)\setminus \T_{z\rightarrow V_{n+m}},$$
where $\T_{z\rightarrow V_{n+m}}$ denotes the connected component of $\T \setminus \{z\}$ containing $V_{n+m}$. 
Thus, 
$$\P(h_{V_n}(V_{n+m+1})\le h_{V_n}(V_{n+m})-i \mid \mathcal{F}_{n+m})\le \frac{\sharp \big(S(z, R_{n+m+1}-i)\setminus \T_{z\rightarrow V_{n+m}}\big)}{\sharp \big(S(V_{n+m},R_{n+m+1})\cap \T_{V_{n+m}}\big)} = K^{-i}.$$

Now, fix $M\ge 1$ such that the law $\mu$ on $\Z$ defined by 
\begin{equation}\label{eq:defimu}
\mu(\{M\})=\frac{K-1}{K} \qquad \mbox{ and } \qquad \mu((-\infty,-i])=K^{-i} \mbox{ for $i\ge 1$}
\end{equation}
has positive expectation. Then, according to Lemma \ref{lem:lienIetR}, on the event $\{I_{n+m+1}\ge K^{M+1}, m\ge 0\}$, we get $R_{n+m+1}\ge M$ for $m\ge 0$ and so, if $h_{V_n}(V_{n+m})\ge 0$, we get the stochastic domination:
$$\mathcal{L}(h_{V_n}(V_{n+m+1})-h_{V_n}(V_{n+m})\mid\mathcal{F}_{n+m})\ge_{st} \mu.$$
Let $(Y_k)_{k\ge 0}$ be an i.i.d.\ sequence with distribution $\mu$. By induction, we get
$$\P(\forall m\ge 0, h_{V_n}(V_{n+m})\ge 0 \mid \mathcal{F}_n)\ge \P(\forall m\ge 0, \sum_{k=0}^m Y_{k}\ge 0)1_{\{I_{n+m+1}\ge K^{M+1}, m\ge 0\}}.$$
Recalling that $h_{V_n}(V_{n+m})\ge 0$ if and only if $V_{n+m}\in \T_{V_n}$ and $\mathcal{H}_n\subset \mathcal{F}_n$, we get 
$$\P(\forall m\ge 0, V_{n+m}\in \T_{V_n} \mid \mathcal{H}_n)\ge \P(\forall m\ge 0, \sum_{k=0}^m Y_{k}\ge 0)\P(I_{n+m+1}\ge K^{M+1}, m\ge 0\mid\mathcal{H}_n).$$
The first probability on the right-hand side is positive since $(Y_k)_{k\ge 0}$ are i.i.d.\ random variables with positive expectation. For the second probability, we have
\begin{eqnarray*}
\P(I_{n+m+1}\ge K^{M+1}, m\ge 0\mid\mathcal{H}_n) &=& \P(I_{n+m+1}\ge K^{M+1}, m\ge 0\mid I_1,\ldots,I_n)\\
&\geq & \P(I_{m+2}\ge K^{M+1},\ m\ge0)>0,
\end{eqnarray*}
where we used Proposition~\ref{prop:Arnold}, Item (iii) for the inequality.
\end{proof}

In order to prove Item 5 of Theorem \ref{thm:GwithoutSeeds}, which states that two distinct connected components of $\mathcal{G}$ cannot share the same asymptotic direction, we need to consider the ancestral lines of two distinct vertices at the same time. To this end, we introduce a coupled exploration process $(U_n,V_n,W_n)_{n\ge 1}$ that simultaneously constructs the ancestral lines $\Gamma(u)$ and $\Gamma(v)$ of two vertices $u,v$ in a specified order extending the construction of the genealogical process of Section \ref{sec:constructionW}. This coupled exploration process can, once again, be encoded in terms of a classical record process.

\medskip
As in Section \ref{sec:constructionW}, we proceed by induction. Set $(U_1,V_1):=(u,v)$ and $W_1:=\{u, v\}$. Assume that $U_n$, $V_n$ and $W_n$ have been constructed. Then, there are three cases: if $\tau_{V_n}>\tau_{U_n}$, we consider the increasing uniform indexing $(\zi_j^{V_n})_{j\ge 1}$ associated to $V_n$ and let 
$$J_n := \inf\{j \ge 1 : \tau_{\zi_j^{V_n}} < \tau_{V_n}\}
     = \inf\{j \ge 1 : \zi_j^{V_n} \notin W_n\setminus\{U_n\},\ \tau_{\zi_j^{V_n}} < \tau_{V_n}\}.$$
We then declare $\zi_{J_n}^{V_n}$ to be the direct ancestor of $V_n$, and we set
$$
V_{n+1} := \zi_{J_n}^{V_n}, \qquad U_{n+1} := U_n,
\qquad
W_{n+1} := W_n \cup \{\zi_j^{V_n} : j \le J_n\}.
$$
It may happen that $V_{n+1}=U_n=U_{n+1}$, which means that the ancestral lines of $v$ and $u$ have just coalesced. In the second case $\tau_{V_n}<\tau_{U_n}$, we perform a symmetric construction interchanging the roles of $V_n$ and $U_n$. Finally, if at some time step $n$ we have $\tau_{V_n}=\tau_{U_n}$, then $V_n=U_n$; the paths have previously merged, and we simply continue with the single exploration process described in the previous section. To summarize this in words: the joint ancestral process of $u$ and $v$ is constructed in the same manner as in Section \ref{sec:constructionW}, and, at each step, we choose to explore the ancestral line of the vertex with the currently largest arrival time. In particular, keeping $V_1$ and then the indices $k\ge 2$ such that $V_k \neq V_{k-1}$ recovers the ancestral line $\Gamma(v)$ of $v$ (and a similar statement holds for $\Gamma(u)$).

\bigskip

Just as for the single genealogical exploration process of Section \ref{sec:constructionW}, it is convenient to encode the joint exploration process $(U_n,V_n,W_n)_{n\ge 1}$ in terms of a record process. But now we need to consider the inter-arrival times $(I'_n,\, n\geq 1)$ of a \emph{lower 2-record process}\footnote{
This means that a new record is reached when we observe a value that is smaller than the second smallest value observed so far. More precisely, let \((X_i)_{i\ge 1}\) be i.i.d.\ with an atomless distribution, set \(T_0:=0\), \(T_1:=2\), and, for \(n\ge 1\), \(T_{n+1}:=\inf\{k> T_n:X_k<\min_2\{X_1,\ldots,X_{T_n}\}\}\), where \(\min_2\) denotes the second smallest element. Then $(T_n)$ is the sequence of record times of the lower 2-record process, and the differences \(I'_n:= T_n- T_{n-1}\) are the inter-arrival times. This corresponds to the \(k=2\) case of the general \(k\)-record process; see \cite{DziubdzielaKopocinski1976,arnold2011}.} instead of a simple record process. We also need a sequence $(\beta_i)_{i\ge 1}$ of i.i.d.\ uniformly distributed random variables on $\{u,v\}$ and a family $((\zi_j^w)_{j \ge 1}, w \in \T)$ of independent uniform increasing indexings. We assume that all these families of random variables are independent.
We define by induction
$$
(\hat{U}_1,\hat{V}_1):=(u,v),
\qquad
\hat{W}_1:=\{u,v\}.
$$
Assume that $(\hat{U}_n,\hat{V}_n, \hat{W}_n)$ has been constructed and that $\hat{U}_n \neq \hat{V}_n$. If $\beta_n=u$, define $\hat{J}_n$ as the first index such that we explore $I'_{n+1}$ new vertices around $\hat{U}_n$ that do not belong to $\hat{W}_n$:
\begin{equation}\label{eq:def1process2record}
\hat{J}_n := \inf\left\{j \ge 1 : \sharp\{i \le j : \zi_i^{\hat{U}_n} \notin \hat{W}_n\}=I'_{n+1}\right\}.
\end{equation}
Let $\hat{L}_n$ denote the index of $\hat{V}_n$ in the increasing indexing around $\hat{U}_n$, \emph{i.e.}, such that $\zi_{\hat{L}_n}^{\hat{U}_n}=\hat{V}_n$. 
\begin{itemize}
    \item If $\hat{J}_n<\hat{L}_n$, then set $\hat{U}_{n+1}=\zi_{\hat{J}_n}^{\hat{U}_{n}}$, $\hat{V}_{n+1}=\hat{V}_n$ and $\hat{W}_{n+1}=\hat{W}_n\cup \{\zi_j^{\hat{U}_{n}}:j\le \hat{J}_n\}$.
    \item Otherwise, set $\hat{U}_{n+1}=\hat{V}_{n+1}=\hat{V}_{n}$ and $\hat{W}_{n+1}=\hat{W}_n\cup \{\zi_j^{\hat{U}_{n}}:j\le \hat{L}_n\}$.
\end{itemize}
If $\beta_n=v$, the construction is the same with the roles of $\hat{U}$ and $\hat{V}$ interchanged.
This constructs the process $(\hat{U}_n,\hat{V}_n,\hat{W}_n)$ up to the first step $n$ (possibly infinite) such that $\hat{U}_n=\hat{V}_n$.

\begin{lemma}\label{lem:memeloi} Consider the two processes defined above. Let $\sigma:=\inf\{n\ge 1, U_n=V_n\}$ and $\hat{\sigma}:=\inf\{n\ge 1, \hat{U}_n=\hat{V}_n\}$. Then $(U_n,V_n,W_n)_{n\le \sigma}$ and $(\hat{U}_n,\hat{V}_n,\hat{W}_n)_{n\le \hat{\sigma}}$ have the same distribution. Moreover, on the event that $\hat{\sigma} > n$, we have
\[
I'_n = \sharp\bigl(\hat{W}_n \setminus \hat{W}_{n-1}\bigr),
\qquad \text{with the convention } \hat{W}_0 := \emptyset.
\]
\end{lemma}

\begin{proof}
We argue by induction before the coalescence time. On the event $\{\sigma>n\}$, the variables $\tau_{U_n}$ and $\tau_{V_n}$ are the two smallest elements of $\{\tau_w:\, w\in W_n\}$. The next step explores the endpoint with the larger of these two arrival times, and new vertices are revealed until either the other endpoint is met or a new vertex has arrival time smaller than $\max(\tau_{U_n},\tau_{V_n})$. Thus, if one ignores the possible coalescence with the other endpoint, the number of new vertices that have to be inspected is exactly an inter-arrival time of a lower 2-record process.

It remains to identify the random choice of which endpoint is explored at each step. For a lower 2-record process, conditionally on the current second record value, say $b$, the current minimum and the new observation which falls below $b$ are independent and have the same law: they are distributed as the arrival-time law conditioned to be smaller than $b$. Since this law is atomless, the new observation is smaller than the current minimum with probability $1/2$, and larger with probability $1/2$. Thus, after each 2-record, the line carrying the larger of the two active arrival times is equally likely to be the $u$-line or the $v$-line. The same conditional independence also gives that these choices are i.i.d.\ and independent of the inter-arrival times and of the uniform increasing indexings. The recursive rules in the construction of $(\hat{U}_n,\hat{V}_n,\hat{W}_n)$ reproduce exactly this mechanism, with the variables $(I'_n)$ giving the 2-record inter-arrival times and the variables $(\beta_n)$ choosing which endpoint is currently the larger one. If the other endpoint is encountered before the next 2-record, both constructions coalesce. This proves the equality in distribution up to the coalescence time. Finally, on $\{\hat{\sigma}>n\}$ no coalescence has occurred before time $n$, and the $n$-th increment of the explored set is precisely the $n$-th 2-record inter-arrival time; in particular $I'_1=2=\sharp\hat{W}_1$ and, for $n\ge2$, $I'_n=\sharp(\hat{W}_n\setminus\hat{W}_{n-1})$.
\end{proof}

\begin{proof}[Proof of Theorem \ref{thm:GwithoutSeeds}, Item 5]
One easily checks that, for any $n \geq 1$, $\hat{W}_n$ either has two connected components in $\T$, one containing $u$ and the other containing $v$, or consists of a unique connected component containing both $u$ and $v$.

For $w\in \T$, define $\T_{w \setminus \{u,v\}}$ to be the subtree of $\T$ composed of $\{w\}$ and all the connected components of $\T\setminus\{w\}$ that contain neither $u$ nor $v$. Observe that $\T_{w \setminus \{u,v\}}$ is a regular $(K+1)$-ary tree except at $w$, whose degree is either $K-1$ or $K$ depending on whether $u$ or $v$ are in the same component of $\T\setminus\{w\}$ or not (see Figure \ref{fig:2direction} for an illustration).

\begin{figure}
    \centering
\IfFileExists{2direction.pdf}{%
\includegraphics[width=11cm]{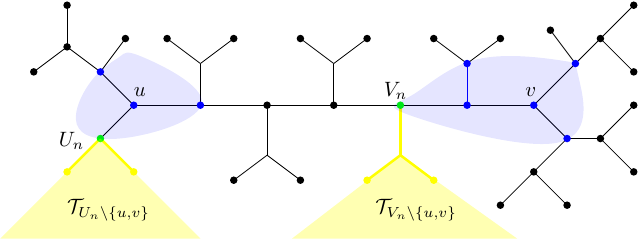}%
}{%
\fbox{\parbox[c][3cm][c]{0.75\linewidth}{\centering\ttfamily Missing figure: 2direction.pdf}}%
}
\caption{\small{Illustration of the double exploration process on a binary tree. The set $\hat W_n$ is represented in blue. The subtrees $\T_{\hat U_n \setminus \{u,v\}}$ and $\T_{\hat V_n \setminus \{u,v\}}$ are in yellow.}}\label{fig:2direction}
\end{figure}

Since each step of the double exploration is performed from the active endpoint with the larger arrival time, the same argument as in the simple exploration process shows that
$$
\hat{W}_n \cap \T_{\hat{U}_n \setminus \{u,v\}} = \{\hat{U}_n\}
\qquad \text{and} \qquad
\hat{W}_n \cap \T_{\hat{V}_n \setminus \{u,v\}} = \{\hat{V}_n\},
$$
and that the sets $\T_{\hat{U}_n \setminus \{u,v\}}$ and $\T_{\hat{V}_n \setminus \{u,v\}}$ are disjoint 
as soon as $\hat{U}_n\neq \hat{V}_n$. Define the filtrations
\begin{eqnarray*}
  \mathcal{F}'_n&:=&\sigma\bigl((I'_k)_{k\ge 1}, (\beta_k)_{k\ge 1}, ((\zi_j^{\hat{U}_i})_{j\ge 1})_{i< n}, ((\zi_j^{\hat{V}_i})_{j\ge 1})_{i< n}\bigr)\\
\mathcal{H}'_n&:=&\sigma((\hat{U}_k,\hat{V}_k,\hat{W}_k), k\le n),
\end{eqnarray*}
and the event $A_n$ that, starting from time $n$, the two explorations remain in their respective sectors:
\[
A_n:=\left\{\forall m\ge 0,\ 
\hat{U}_{n+m}\in \T_{\hat{U}_n \setminus \{u,v\}},\ 
\hat{V}_{n+m}\in \T_{\hat{V}_n \setminus \{u,v\}}\right\}.
\]
Fix $M\ge1$ such that the law $\mu$ defined in \eqref{eq:defimu} has positive expectation. Using an argument similar to the proof of Lemma \ref{lemm:VncondHn}, we now find that
$$\P(\hat{\sigma}<\infty \ \text{or}\ A_n \mid \mathcal{F}'_n) \;\ge \; \P\Big(\forall m\ge 1, \sum_{k=0}^m Y'_{k} > 0\Big)^2 \; 1_{\{I'_{n+m+1} \ge K^{M+1}, m\ge 0\}},$$
where $(Y'_{k})_{k\ge 0}$ are independent random variables\footnote{$Y'_0$ has this particular distribution because the vertex $\hat{U}_n$ may have degree $K-1$ in the tree $\T_{\hat{U}_n \setminus \{u,v\}}$.} such that
$$\P(Y'_0=M)=\frac{K-1}{K+1} \qquad \mbox{ and } \qquad \P(Y'_0=-1)=\frac{2}{K+1},$$
and, for $k\ge 1$, $Y'_{k}$ has law $\mu$ defined by \eqref{eq:defimu}. We have $\P(\forall m\ge0,\sum_{k=0}^mY'_k\ge0) >0$ because $\P(Y'_0=M)>0$ and the subsequent random walk has increment law $\mu$ with positive drift.
The proof of Proposition~\ref{prop:Arnold} adapts directly to the lower 2-record process and gives
$$\P(I'_{n+m+1} \ge K^{M+1}, m\ge 0\mid \mathcal{H}'_n) \;\ge\; \P( I'_{m+2} \ge K^{M+1}, m\ge 0) \;>\; 0.$$
Thus, we deduce the existence of $c'>0$ such that, for any $n\ge 1$, a.s.
$$\P(\hat{\sigma}<\infty \ \text{or}\ A_n\mid \mathcal{H}'_n)\ge c'.$$
Finally, a standard restart argument at the successive exit times from the current sectors gives that
$$\P(\hat{\sigma}<\infty \ \text{or}\ A_n \hbox{ for some $n$})= 1.$$
By Lemma \ref{lem:memeloi}, the same dichotomy holds for the original coupled exploration process. So either $\sigma < \infty$ and the two ancestral paths $\Gamma(u)$ and $\Gamma(v)$ eventually coalesce, or they eventually remain in disjoint subtrees, and therefore $\dir \Gamma(u)\neq \dir \Gamma(v)$.
\end{proof}

\subsection{Proof of Theorem \ref{thm:GwithoutSeeds}, Item 2 }

Let $u\in \T$. Recall that $p(u)$ denotes the direct ancestor of $u$ in the genealogical graph (\emph{i.e.}, the vertex from which $u$ inherits its color) and that $D(u)$ stands for the set of descendants of $u$ (\emph{i.e.}, the set of vertices whose color inheritance can be traced back to $u$). Rigorously, this means that, given $v\in D(u)$, there exists a finite sequence of vertices $(x_0,x_1,\ldots, x_n)$ such that 
    \begin{itemize}
        \item[(i)] $x_0=u$ and $x_n=v$;
        \item[(ii)] for $0\le i \le n-1$, $p(x_{i+1})=x_i$.
    \end{itemize}
Points (i) and (ii) imply in particular that
\begin{itemize}
    \item[(a)] $\tau_u<\tau_{x_1}<\ldots<\tau_{x_{n-1}}<\tau_v$;
        \item[(b)] for $0\le i \le n-1$, the ball $B(x_{i+1},d(x_i,x_{i+1})-1)$ contains no vertices with arrival time smaller than $\tau_{x_{i+1}} > \tau_u$. 
    \end{itemize}
In this section, we prove that $D(u)$ is a.s. finite. To do so, we upper bound the probability that points (a) and (b) occur. The difficulty is that the balls in point (b) may overlap heavily. We therefore extract below a subsequence $(x_{i_1},\ldots,x_{i_m})$ and integers $r_1,\ldots,r_m$ such that the balls $B(x_{i_j},r_j-1)$ are pairwise disjoint, contain no colored vertex at time $\tau_u$, and still control the total displacement from $u$ to $v$. 

To make this idea precise, we use the following notation: for $u,v\in \T$, $m\ge 1$ and positive integers $r_1,\ldots,r_m$, we define $A(u,v,m,r_1,\ldots,r_m)$ to be the set of sequences of vertices $(\bar{x}_1,\ldots,\bar{x}_m)$ for which one can find vertices $(v_1,\ldots, v_{m-1})$ on the path $[u,v]$ such that, setting $v_0:=u$, the vertices $v_0,v_1,\ldots,v_{m-1}$ are ordered by strictly increasing distance from $u$ and we have
    \begin{itemize}
        \item $d(\bar{x}_i,v_i)=d(\bar{x}_i,v_{i-1})=r_i$ for $1\le i\le m-1$, 
        \item $d(\bar{x}_m,v_{m-1})=r_m$ and $d(\bar{x}_m,v)\le r_m$.
    \end{itemize}
See Figure \ref{fig:chemin} for an illustration.
\begin{figure}
    \centering
\includegraphics[width=11cm]{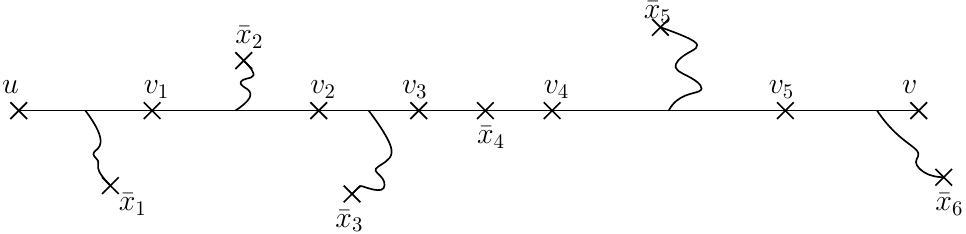}
\caption{\small{Illustration of a sequence $(\bar{x}_1,\ldots,\bar{x}_6)\in A(u,v,6,r_1,\ldots,r_6)$.}}\label{fig:chemin}
\end{figure}

\begin{lemma}\label{lemm:descendance1} Let $v\in D(u)\setminus\{u\}$ and let $(x_0:=u,x_1,\dots, x_n:=v)$ be such that $p(x_{i+1})=x_i$ for every $0\le i\le n-1$. Then there exist $m\ge 1$, integers $r_1,\ldots,r_m\ge 1$, and indices $1\le i_1<\ldots <i_m\le n$ such that 
\begin{itemize}
    \item $(x_{i_1},\ldots, x_{i_m})\in A(u,v,m,r_1,\ldots,r_m)$.
    \item At time $\tau_u$, the union of balls $\bigcup_{j=1}^m B(x_{i_j},r_j-1)$ does not contain a colored vertex.
\end{itemize}

\end{lemma}


\begin{proof} 
We construct inductively, for $j\ge 1$, a subsequence $(x_{i_1},\ldots, x_{i_j})$, a sequence of positive integers $(r_1,\ldots,r_j)$ and a sequence of vertices $(v_1,\ldots, v_{j-1})$ such that
\begin{itemize}
\item[(i)] The vertices $v_1,\ldots,v_{j-1}$ lie on the path $]u,v[$ and are ordered by increasing distance from $u$.
 \item[(ii)] The vertices $v_1,\ldots,v_{j-1}$ witness that $(x_{i_1},\ldots, x_{i_j})\in A(u,x_{i_j},j,r_1,\ldots,r_j)$.
    \item[(iii)] At time $\tau_u$, the union of balls $\bigcup_{h=1}^j B(x_{i_h},r_h-1)$ contains no colored vertex.
\end{itemize}

Let us start with $j=1$. Removing the $K+1$ edges around $u$, we obtain $K+1$ disjoint subtrees of $\T$. We denote by $\T_{u\to v}$ the one containing the vertex $v$. Observe that $x_1\in \T_{u\to v}$. Indeed, set
$$I:=\inf\{i\ge 1, x_i\in \T_{u\to v}\},$$
which is well defined since $x_n=v\in \T_{u\to v}$ and hence $I\le n$. Moreover, since $x_{I}$ inherited its color from $x_{I-1}$, this implies that $d(x_I,x_{I-1})\le d(x_I,u)$. So, either $x_{I-1}=u$, \emph{i.e.}, $x_1\in \T_{u\to v}$, or $x_{I-1}\in \T_{u\to v}$, but this would contradict the definition of $I$. Therefore $I=1$. We set
$$i_1=1 \qquad r_1:=d(u,x_{1}).$$
 Of course, we have $$(x_{i_1})\in A(u,x_{i_1},1,r_1).$$
Since $x_{1}$ inherits its color from $u=x_0$, it implies that the ball $B(x_{i_1},r_1-1)$ contains no colored vertices at time $\tau_u$. If $d(x_1,v)\le r_1$, we set $m=1$ and the sequence $(x_1)$ satisfies the condition of Lemma \ref{lemm:descendance1}.

\medskip

Assume now that $(x_{i_1},\ldots, x_{i_j})$, a sequence of positive integers $(r_1,\ldots,r_j)$ and a sequence of vertices $(v_1,\ldots, v_{j-1})$ have been constructed such that (i), (ii), (iii) hold. 
If $d(v,x_{i_j})\le r_j$, then we set $m=j$ and we have as required
\begin{itemize}
    \item $(x_{i_1},\ldots, x_{i_m})\in A(u,v,m,r_1,\ldots,r_m)$.
    \item At time $\tau_u$, the union of balls $\bigcup_{h=1}^m B(x_{i_h},r_h-1)$ does not contain a colored vertex.
\end{itemize}
Otherwise, $d(v,x_{i_j})> r_j$ and, with the convention $v_0:=u$, we define $v_j$ as the unique vertex on $]v_{j-1},v[$ such that $d(x_{i_j},v_j)=r_j$. Removing the $K+1$ edges around $v_j$, we obtain $K+1$ disjoint subtrees of $\T$ and we denote by $\T_{v_j\to v}$ the one containing the vertex $v$. 
Let 
$$i_{j+1}:=\inf\{i> i_j, x_i\in \T_{v_j\to v}\} \qquad r_{j+1}:=d(v_j,x_{i_{j+1}}).$$
As before, note that $x_n=v\in \T_{v_j\to v}$ so $i_{j+1}$ is well defined and $i_{j+1}\le n$. Moreover, note that 
$$d(x_{i_{j+1}},x_{i_{j+1}-1}) \ge d(x_{i_{j+1}},v_j)=r_{j+1}.$$
Since $x_{i_{j+1}}$ inherited its color from $x_{i_{j+1}-1}$, this implies that the ball $B(x_{i_{j+1}},r_{j+1}-1)$ contains no colored vertex at time $\tau_u$. Moreover, by construction
$(x_{i_1},\ldots, x_{i_{j+1}})\in A(u,x_{i_{j+1}},j+1,r_1,\ldots,r_{j+1})$.
Together with the induction hypothesis, this gives (i)--(iii) at rank $j+1$. Since the indices \(i_j\) are strictly increasing and are bounded by \(n\), the induction stops after finitely many steps.
\end{proof}

\begin{proof}[Proof of Theorem \ref{thm:GwithoutSeeds}, Item 2]
Fix $u,v\in \T$ with $u\neq v$ and let $\ell:=d(u,v)$. To lighten the notation, we use $\bar{\mathbf{x}}_m$ to denote a sequence of vertices $(\bar{x}_1,\ldots,\bar{x}_m)$ and $\mathbf{r}_m$ to denote a sequence of positive integers $(r_1,\ldots,r_m)$. Let us also define the event
$$\mathcal{B}(\bar{\mathbf{x}}_m,\mathbf{r}_m):=\{\hbox{At time $\tau_u$, $\bigcup_{j=1}^m B(\bar{x}_{j},r_j-1)\setminus\{\bar{x}_1,\ldots,\bar{x}_m\}$ does not contain a colored vertex}\}.$$
Using the previous lemma, we know that if $v\in D(u)$, there exist $m\ge 1$, $\mathbf{r}_m\ge 1$, $\bar{\mathbf{x}}_{m}\in A(u,v,m,\mathbf{r}_m)$ such that 
\begin{itemize}
    \item $\tau_{\bar{x}_1}<\ldots<\tau_{\bar{x}_m}$;
    \item $\mathcal{B}(\bar{\mathbf{x}}_m,\mathbf{r}_m)$ occurs.
\end{itemize}
For fixed $\bar{\mathbf{x}}_m$, conditionally on $\tau_u$, the event $\mathcal{B}(\bar{\mathbf{x}}_m,\mathbf{r}_m)$ only involves arrival times outside $\{\bar{x}_1,\ldots,\bar{x}_m\}$, whereas the event $\{\tau_{\bar{x}_1}<\ldots<\tau_{\bar{x}_m}\}$ only involves the arrival times of $\bar{x}_1,\ldots,\bar{x}_m$. These two events are therefore independent, and the latter has probability $1/m!$. Thus, we get 
\begin{eqnarray}    \P(v\in D(u)\mid \tau_u) & \le & \sum_{m\ge 1} \sum_{\mathbf{r}_m\ge 1}  \sum_{\bar{\mathbf{x}}_m\in A(u,v,m,\mathbf{r}_m)}\P(\tau_{\bar{x}_1}<\ldots<\tau_{\bar{x}_m},\, \mathcal{B}(\bar{\mathbf{x}}_m,\mathbf{r}_m) \mid \tau_u) \nonumber \\
& = & \sum_{m\ge 1} \sum_{\mathbf{r}_m\ge 1}  \sum_{\bar{\mathbf{x}}_m\in A(u,v,m,\mathbf{r}_m)} \frac{1}{m!}\P(\mathcal{B}(\bar{\mathbf{x}}_m,\mathbf{r}_m) \mid \tau_u). \label{eq:descendance1}
\end{eqnarray}

Let us note that if $\bar{\mathbf{x}}_m\in A(u,v,m,\mathbf{r}_m)$, the balls $ (B(\bar{x}_{j},r_j-1))_{1\le j\le m}$ are pairwise disjoint. Indeed, for $i<j$, the ball $B(\bar{x}_{i},r_i-1)$ is contained in $\T_{v_i\to u}$
whereas $B(\bar{x}_{j},r_j-1)$ is contained in $\T_{v_i\to v}$. Moreover, we have the rough lower bound $\sharp B(\bar{x}_{i},r_i-1)\ge K^{r_i-1}$ so we get 
$$ \sharp \Big(\bigcup_{j=1}^m B(\bar{x}_{j},r_j-1)\setminus\{\bar{x}_1,\ldots,\bar{x}_m\}\Big) \ge \sum_{i=1}^m K^{r_i-1}-m.$$
This yields
\begin{equation}\label{eq:descendance2}
  \P( \mathcal{B}(\bar{\mathbf{x}}_m,\mathbf{r}_m) \mid \tau_u) \,\le\, \P( \tau>\tau_u \mid \tau_u)^{\sum_{i=1}^m K^{r_i-1}-m}.
\end{equation}
We now give an upper bound on the number of sequences $\bar{\mathbf{x}}_m\in A(u,v,m,\mathbf{r}_m)$. First note that 
$$\ell=d(u,v)\le 2\sum_{i=1}^m r_i.$$
Hence 
\begin{equation}\label{eq:descendance3}
  \sharp A(u,v,m,\mathbf{r}_m)=0 \quad \hbox{ if }  2\sum_{i=1}^m r_i<\ell.
\end{equation}
Otherwise, by definition, $\bar{x}_1$ must lie in the branch from $u$ toward $v$, so, since $d(\bar{x}_1,u)=r_1$, there are at most $K^{r_1}$ choices for $\bar{x}_1$. Once the vertex $\bar{x}_1$ is chosen, $v_1$ is entirely determined. The same argument applies recursively from $v_i$ toward $v$. Thus, using that $d(\bar{x}_2,v_1)=r_2$, we obtain that there remain at most $K^{r_2}$ choices for $\bar{x}_2$ and so on. This yields
\begin{equation}\label{eq:descendance4}
  \sharp A(u,v,m,\mathbf{r}_m) \le K^{\sum_{i=1}^m r_i}.
\end{equation}
As already remarked, the law of the arrival times does not matter (as long as it is atomless), so we may assume here that $(\tau_w)_{w\in \T}$ are i.i.d.\ uniform in $[0,1]$. Combining \eqref{eq:descendance1}, \eqref{eq:descendance2}, \eqref{eq:descendance3} and \eqref{eq:descendance4}, we get
 \begin{eqnarray*}    \P(v\in D(u)\mid \tau_u) 
&\le & \sum_{m\ge 1} \sum_{\substack{\mathbf{r}_m\ge 1\\ 2\sum_{i=1}^m r_i\ge \ell}}K^{\sum_{i=1}^m r_i} \frac{1}{m!} (1-\tau_u)^{\sum_{i=1}^m K^{r_i-1}-m}\\
&=&\sum_{m\ge 1} \sum_{k\ge \ell/2m}\sum_{\substack{||\mathbf{r}_m||_{\infty}=k \\ 2||\mathbf{r}_m||_{1}\ge \ell}}\frac{1}{m!}\prod_{i=1}^m\left(K^{ r_i}  (1-\tau_u)^{ K^{r_i-1}-1}\right).
\end{eqnarray*}
Summing this upper bound over all the vertices $v \in \T$, we get
   \begin{eqnarray*}  \E[\sharp D(u)\mid \tau_u]
   &=& 1+\sum_{v\in \T\setminus\{u\}}{\P(v\in D(u)\mid \tau_u)}\\
   &\le & 1+ (K+1)\sum_{\ell \ge 1} K^\ell  \sum_{m\ge 1} \sum_{k\ge \ell/2m}\sum_{\substack{||\mathbf{r}_m||_{\infty}=k \\ 2||\mathbf{r}_m||_{1}\ge \ell}}\frac{1}{m!}\prod_{i=1}^m\left(K^{ r_i}  (1-\tau_u)^{ K^{r_i-1}-1}\right)\\
    &\le & 1+ (K+1)\sum_{\ell \ge 1}   \sum_{m\ge 1} \sum_{k\ge \ell/2m}\sum_{\substack{||\mathbf{r}_m||_{\infty}=k \\ 2||\mathbf{r}_m||_{1}\ge \ell}}K^{2||\mathbf{r}_m||_{1}}\frac{1}{m!}\prod_{i=1}^m\left(K^{ r_i}  (1-\tau_u)^{ K^{r_i-1}-1}\right)\\
    &=& 1+ (K+1)\sum_{\ell \ge 1}   \sum_{m\ge 1} \sum_{k\ge \ell/2m}\sum_{\substack{||\mathbf{r}_m||_{\infty}=k \\ 2||\mathbf{r}_m||_{1}\ge \ell}}\frac{1}{m!}\prod_{i=1}^m\left(K^{ 3r_i}  (1-\tau_u)^{ K^{r_i-1}-1}\right).
\end{eqnarray*}
Note that for any $\tau_u\in (0,1)$, the function $z\mapsto K^{3z}(1-\tau_u)^{K^{z-1}-1}$ is bounded on $[0,\infty)$ by a constant $C\ge 1$ (which depends on $\tau_u$). Therefore, we get
 \begin{eqnarray*}  \E[\sharp D(u)\mid \tau_u]
    &\le & 1+(K+1)\sum_{\ell \ge 1}   \sum_{m\ge 1} \sum_{k\ge \ell/2m}\sum_{||\mathbf{r}_m||_{\infty}=k }\frac{C^m}{m!}\left(K^{ 3k}  (1-\tau_u)^{ K^{k-1}-1}\right)\\
    &\le & 1+(K+1)\sum_{k\ge 1}   \sum_{m\ge 1} \sum_{\ell\le 2mk}\frac{(kC)^m}{m!}\left(K^{ 3k}  (1-\tau_u)^{ K^{k-1}-1}\right)\\
    &= & 1+2(K+1)\sum_{k\ge 1}   \sum_{m\ge 1} k\frac{(kC)^m}{(m-1)!}\left(K^{ 3k}  (1-\tau_u)^{ K^{k-1}-1}\right)\\
     &\le & 1 + 2(K+1)\sum_{k\ge 1}  k^2C\left(e^{C}K^{3}\right)^{k}  (1-\tau_u)^{ K^{k-1}-1}\,<\, \infty.
\end{eqnarray*}
This proves that \(\E[\sharp D(u)\mid \tau_u]<\infty\), and hence \(D(u)\) is finite a.s. Since \(\T\) is countable, the conclusion holds simultaneously for every \(u\in\T\).
\end{proof}

\begin{remark} The argument of the proof above can be refined to give a faster-than-polynomial tail for the size of the descendant cluster and prove, in particular, that $\E[(\sharp D(u))^p\mid \tau_u]<\infty$ a.s.\ for every $p$.
\end{remark}

\section{The genealogical graph with seeds}\label{s:Gwithseed}

This section is devoted to the proof of Theorem \ref{thm:GwithSeeds}. Most statements of the theorem are consequences of Theorem \ref{thm:GwithoutSeeds} together with the preliminary results obtained in Section \ref{s:genealogicalprocess}. The main remaining work is to determine when a seed has infinitely many children. 

\medskip

Consider a finite set \( S := \{\tilde{s}_1, \ldots, \tilde{s}_M\} \) of seeds.
We construct the genealogical graph without seeds, \( \mathcal{G} \), and the one with seeds, \( \tilde{\mathcal{G}} \), on the same probability space using:
\begin{itemize}
    \item the same arrival times \( \tau_v \) for all \( v \notin S \),
    \item the same random indexing \( (\zi_j^v)_{j \ge 1} \) for every \( v \in \T \).
\end{itemize}
For \( v \in \T \), we still denote by \( \Gamma(v) = (V_i(v))_{i \ge 1} \) its ancestral path in \( \mathcal{G} \), and by \( W_n(v) \) the set of vertices explored to determine \( V_n(v) \). We similarly denote by \( \tilde{\Gamma}(v) = (\tilde{V}_i(v))_{i \ge 1} \) the ancestral path of \( v \) in \( \tilde{\mathcal{G}} \). For \(v\notin S\), we have the following dichotomy:
\begin{itemize}
    \item If $W_\infty(v)\cap S =\emptyset$, then $\tilde{\Gamma}(v)=\Gamma(v)$, meaning that the ancestral path $\tilde{\Gamma}(v)$ of $v$ is infinite or, equivalently, that this vertex does not belong to the genealogical component of a seed in \( \tilde{\mathcal{G}} \).
    \item Otherwise, let $n_0:=\inf\{n\ge 1:\; W_n(v)\cap S\neq \emptyset\}$. Then $\tilde{\Gamma}(v)=(V_1,\ldots,V_{n_0-1},s)$, where $s$ is the first seed encountered. In this case, $v$ belongs to the genealogical component of the seed $s$ in \( \tilde{\mathcal{G}} \) and its ancestral path $\tilde{\Gamma}(v)$ is finite.
\end{itemize}
Moreover, if $v$ is not a seed, the set of descendants $\tilde{D}(v)$ of $v$ in $\tilde{\mathcal{G}}$ is a subset of $D(v)$, the set of descendants of $v$ in $\mathcal{G}$.
We deduce that, if we write the decomposition of $\mathcal G$ into its connected components as
$$
\mathcal G = \bigsqcup_{i \in I} \mathcal{C}_i,
$$
At this stage, some components of $\mathcal G$ may a priori be entirely captured by seed components. Thus, the decomposition of $\tilde{\mathcal{G}}$ into connected components satisfies
$$
\tilde{\mathcal G} = \bigsqcup_{i \in \tilde I} \tilde{\mathcal{C}}_i \sqcup \bigsqcup_{j = 1}^M\tilde{\mathcal{U}}_j
$$
where 
\begin{itemize}
    \item $\tilde I\subset I$ and $\tilde{\mathcal{C}}_i$ is an infinite oriented subtree of $\mathcal{C}_i$.
    \item For each $j\in \{1,\ldots, M\}$, $\tilde{\mathcal{U}}_j$ is an oriented tree rooted at $\tilde{s}_j$. Moreover, for $v \in \tilde{\mathcal{U}}_j$, $v \neq \tilde{s}_j$, we have $\E[\sharp\tilde{D}(v) \mid \tau_v] \le \E[\sharp D(v) \mid \tau_v] <\infty$. In particular, we have
    $$
    \hbox{$\tilde{\mathcal{U}}_j$ is finite} \quad \Longleftrightarrow \quad \hbox{$\tilde{s}_j$ has finite degree in $\tilde{\mathcal{U}}_j$}\qquad\hbox{a.s.}
    $$    
\end{itemize}
In order to complete the proof of Theorem \ref{thm:GwithSeeds}, it remains to show
that
\begin{enumerate}
    \item[(i)] For every $i\in I$, there exists $v\in \mathcal{C}_i$ such that $\Gamma(v)$ does not end at a seed, \emph{i.e.},
    \begin{equation}\label{eq:theo3_1}
        \tilde I=I
    \end{equation} 
    \item[(ii)] We have
    \begin{equation}\label{eq:theo3_2}
        \P(\mbox{$\tilde{s}_j$ has infinite degree in $\tilde{\mathcal{U}}_j$})=
\begin{cases}
    1 & \mbox{if $\{y\in \T:\; d(y,\tilde{s}_j)=\min_{i\le M} d(y,\tilde{s}_i)\}$ is infinite}\\
    0 & \mbox{otherwise}
\end{cases}
    \end{equation} 
\end{enumerate}

\begin{figure}
\begin{center}
\begin{tabular}{c@{\hspace{3cm}}c}
\includegraphics[width=5cm]{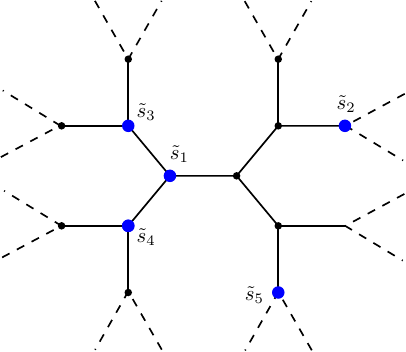} &
\includegraphics[width=5cm]{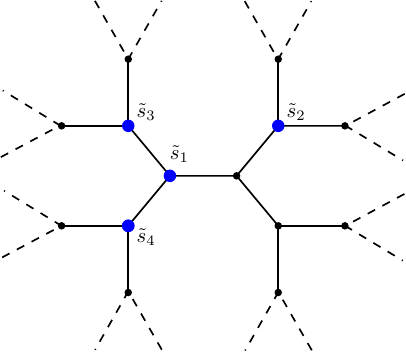}
\end{tabular}
\end{center}
\caption{Examples of seed configurations on the binary tree. Seeds are shown in blue. On the left, the seed $\tilde{s}_1$ has almost surely a finite set of descendants. On the right, the seed $\tilde{s}_1$ has almost surely infinite degree in $\tilde{\mathcal{G}}$.}
\label{Fig:cas_seed}
\end{figure}

\begin{proof}[Proof of \eqref{eq:theo3_1}] Fix $i\in I$ and $v\in \mathcal{C}_i$, and let $\Gamma(v):=(V_n)_{n\ge 1}$ be its ancestral path in the genealogical graph without seeds. We prove that, for $n$ large enough, $W_\infty(V_n)\cap S=\emptyset$, and hence $V_n\in \tilde{\mathcal{C}}_i \neq \emptyset$.

Recall that $R_{n+1}$ denotes the distance between $V_n$ and $V_{n+1}$. By Corollary \ref{coro:V_nenn2}, it holds that, almost surely,
$$ \log(K)R_n\le 2n \quad \mbox{ and } \quad
|V_{n+1}|\ge  \frac{n^2}{3\log K} \qquad\hbox{for $n$ large enough.}$$
Using that the set of seeds $S$ is finite, we deduce that, for $N$ large enough,
\[
S\cap \bigcup_{n\ge N}B(V_{n},R_{n+1})=\emptyset.
\]
Since $W_\infty(V_{N})\subset \bigcup_{n\ge N}B(V_{n},R_{n+1})$, this implies that $W_\infty(V_N)\cap S=\emptyset$, and therefore $V_N\in \tilde{\mathcal{C}}_i$.
\end{proof}

\begin{proof}[Proof of \eqref{eq:theo3_2}] For $v\in \T\setminus S$, recall that
$\tilde{p}(v)$ denotes the direct ancestor of $v$ in $\tilde{\mathcal{G}}$. Let 
$$\tilde{D}_1(\tilde{s}_j):=\{v\in \T:\; \tilde{p}(v)=\tilde{s}_j\}$$
denote the set of direct descendants of the seed $\tilde{s}_j$ in $\tilde{\mathcal{G}}$. 
Let
$$
\T':=\{u\in \T:\; d(u,\tilde{s}_j)=\min_{i\le M} d(u,\tilde{s}_i)\}.
$$
Note that $\tilde{D}_1(\tilde{s}_j)\subset \T'$ because if \( v \notin \T' \), then there exists another seed closer to \( v \) than \( \tilde{s}_j \), so it cannot happen that \( \tilde{p}(v) = \tilde{s}_j \). Hence 
$$\T' \mbox{ finite }\Longrightarrow \P(\tilde{s}_j \mbox{ has infinite degree in }\tilde{\mathcal{U}}_j )=0.$$
This shows the second case of \eqref{eq:theo3_2}. Let us now assume that $\T'$ is infinite. Note that $\T'$ is a subtree of $\T$ and, because it is infinite, there exists a vertex $o'\in \T$ such that $\T_{o'} \subset \T'$ and we can furthermore assume w.l.o.g. $\tilde{s}_j\notin \T_{o'}$. We prove that
\begin{equation}\label{eq:seed_infini}
    \T_{o'} \cap \tilde{D}_1(\tilde{s}_j) \hbox{ is infinite a.s.}
\end{equation}

Let $h:=d(\tilde{s}_j,o')$. For $n\ge 1$, let $u_n$ be the vertex with minimal arrival time among the vertices $v$ of $\T_{o'}$ such that $d(o',v)=n$. Define the event
$$
\mathcal{E}_n:=\left\{\tau_{u_n}\le \tau_v \text{ for all } v \text{ such that } d(v,u_n)\le n+h \text{ and } d(v,o')\ge \frac{n}{3}\right\}.
$$

Note that $d(u_n,\tilde{s}_j)=n+h$, so the ancestor of $u_n$ is at distance at most $n+h$ from $u_n$. Hence, if $\mathcal{E}_n$ occurs, then
$$
\tilde{p}(u_n)\in \{v\in \T_{o'}:\ d(o',v)<n/3\}\cup \{v\in \T:\ d(o',v)\le h\}.
$$

We now prove the following two facts:
\begin{itemize}
    \item[(i)] Infinitely many events $\mathcal{E}_n$ occur almost surely.
    \item[(ii)] $\tilde{p}(u_n)\notin \{v\in \T_{o'}:\ d(o',v)<n/3\}$ for all $n$ large enough almost surely.
\end{itemize}

\begin{figure}
    \centering
\includegraphics[width=7cm]{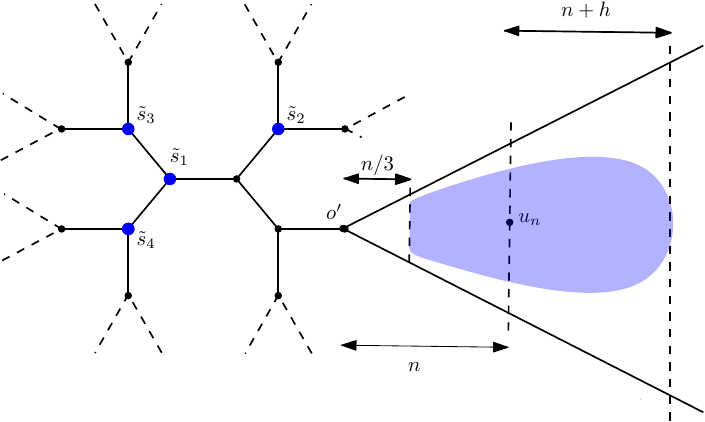}
\caption{\small{Illustration of the notation used in the proof of \eqref{eq:theo3_2} for the seed $\tilde{s}_1$ in the configuration on the right of Figure \ref{Fig:cas_seed}. On the event $\mathcal{E}_n$, no vertex in the blue region has been colored before $u_n$. For $n$ large enough, if this event occurs, the ancestor of $u_n$ is either $\tilde{s}_1$ or $\tilde{s}_2$, each with probability $1/2$.}}
\label{fig:infiniteseed}
\end{figure}

On the one hand, we note that
$$
\sharp\left\{v:\ d(v,u_n)\le n+h \text{ and } d(v,o')\ge \frac{n}{3}\right\}
\le \sharp B(u_n,n+h)
\le (K+1)K^{n+h}.
$$
Since $\tau_{u_n}$ is the minimum of $K^n$ i.i.d.\ random variables, the bound above yields
$$
\P(\mathcal{E}_n) \ge \frac{K^n}{K^n+(K+1)K^{n+h}} = \frac{1}{1+(K+1)K^h} =:c_0>0.
$$
Moreover, for $n>3h$, the event $\mathcal{E}_n$ only depends on the arrival times of the vertices $v$ of $\T_{o'}$ such that
$$
\frac{n}{3}\le d(o',v)<3n.
$$
Thus, the events $(\mathcal{E}_{9^n})_{n>3h}$ are independent and the Borel--Cantelli lemma implies that (i) holds.

On the other hand, a similar computation using that $\sharp\{v:\ d(o',v)<n/3\} \le (K+1)K^{\lfloor n/3 \rfloor}$ shows that
$$
\P\left(\min \{\tau_v:\ d(o',v)<n/3\}\le \tau_{u_n}\right)
\le
\frac{(K+1)K^{n/3}}{K^n+(K+1)K^{n/3}}
\le
(K+1)K^{-2n/3}.
$$
Hence, the Borel--Cantelli lemma now implies that
$$
\min \{\tau_v:\ d(o',v)<n/3\}>\tau_{u_n}
$$
for all $n$ large enough. This implies that the ancestor of $u_n$ cannot belong to $\{v\in \T_{o'}:\ d(o',v)<n/3\}$, so (ii) holds. 

Let us now conclude that \eqref{eq:seed_infini} holds.
Using (i) and (ii), we find that, for infinitely many $n$,
$$
\tilde{p}(u_n)\in \{v\in \T:\ d(o',v)\le h\}.
$$
Since $\tau_{u_n}$ tends to $0$ almost surely, it follows that, ultimately, 
$\tilde{p}(u_n)$ must be a seed, \emph{i.e.}
$$
\tilde{p}(u_n)\in \{s\in S:\ d(o',s)=d(o',\tilde{s}_j)\}\qquad\hbox{infinitely often a.s.}
$$
Finally, by the independence of the tie-breaking indexings at the distinct vertices $u_n$, each seed in this finite set is chosen infinitely often; in particular, $\tilde{s}_j$ has infinitely many direct descendants a.s.
\end{proof}

\section{Local limits}\label{s:local_limits}

We now come back to the case of finite trees and study the local limits of the random coloring in different settings. Recall that $\T^{(\ell)} := \T \cap B(o,\ell)$ denotes the finite tree of height $\ell$ rooted at $o$.
\medskip

Let $s_1,\ldots, s_m$ be a (possibly empty) set of seeds. We always assume $\ell$ is sufficiently large so that $s_i \in \T^{(\ell)}$ for all $i$. We can construct the genealogical graph $\tilde{\mathcal{G}}$ on $\T$ and the genealogical graph $\tilde{\mathcal{G}}^{(\ell)}$ on the finite tree $\T^{(\ell)}$ (both with seeds $s_1,\ldots, s_m$) on the same probability space using the natural coupling: we use the same arrival times $(\tau_v)_{v\in \T}$ and the same uniform increasing indexings $((\zi_j^w)_{j\ge 1}, w\in \T)$ to break ties (and we simply ignore the vertices not in $\T^{(\ell)}$ when constructing $\tilde{\mathcal{G}}^{(\ell)}$).
\medskip

By definition, the genealogical graph is constructed by a local exploration around each vertex of the tree, which makes it nicely behaved with respect to the notion of local convergence. More precisely, we have the straightforward result, stated here for trees but valid for any locally finite connected graph:

\begin{prop} We have
$$
\tilde{\mathcal{G}}^{(\ell)} \underset{\ell\to\infty}{\longrightarrow}  \tilde{\mathcal{G}}
$$
for the topology of local convergence. In other words, for any $N\geq 0$ we have
$$
\tilde{\mathcal{G}} \cap B(o,N) \; = \; \tilde{\mathcal{G}}^{(\ell)} \cap B(o,N) \qquad\hbox{for all $\ell$ large enough a.s.}
$$
where $\tilde{\mathcal{G}} \cap B(o,N)$ (resp.\ $\tilde{\mathcal{G}}^{(\ell)} \cap B(o,N)$) stands for the genealogical graph $\tilde{\mathcal{G}}$ (resp.\ $\tilde{\mathcal{G}}^{(\ell)}$) restricted to the set of vertices $B(o,N)$.
\end{prop}

\begin{proof}
We just need to prove that any vertex $v$ (that is not a seed) has the same direct ancestor in $\tilde{\mathcal{G}}^{(\ell)}$ and $\tilde{\mathcal{G}}$ for $\ell$ large enough. Since both graphs are constructed with the same arrival times and the same increasing indexing, this is true as soon as the exploration around $v$ does not reach the boundary of $\T^{(\ell)}$, \emph{i.e.}, as soon as $\ell \geq d(o, v) + d(v, \tilde{p}(v))$, where $\tilde{p}(v)$ denotes the direct ancestor of $v$ in $\tilde{\mathcal{G}}$. 
\end{proof}

While the genealogical graph is continuous with respect to the local topology, the coloring itself does not have this property because the color of a vertex is related to the connectedness of vertices in the genealogical graph and therefore to the existence of an infinite ancestral path, which is a non-local notion. In the rest of this section, we study the limits of the coloring in two particular cases: 
\begin{itemize}
\item[(a)] Two fixed seeds colored blue and red, respectively. 
\item[(b)] A single blue seed at the root and a red seed on each leaf of $\T^{(\ell)}$. 
\end{itemize}

\subsection{Local limit with two seeds (Theorem \ref{theo:twoseed})}

We fix another vertex $o'$ and consider the random coloring on $\T^{(\ell)}$ with a blue seed at the root $o$ and a red seed at $o'$. This gives a partition of the vertices of $\T^{(\ell)}$ into the blue set $\tilde{B}^{(\ell)}$ and the red set $\tilde{R}^{(\ell)}$, which are precisely the vertices of the two connected components of the genealogical graph $\tilde{\mathcal{G}}^{(\ell)}$ rooted respectively at $o$ and $o'$.

\medskip

As in the previous sections, we denote by $\tilde{\Gamma}(v)=(\tilde{V}_i)_{i\le n(v)}$ the ancestral path of $v$ in $\tilde{\mathcal{G}}$. According to Theorem \ref{thm:GwithSeeds}, this path may be infinite ($n(v)=\infty$) or finite ($n(v) < \infty$) and, in the latter case, it ends either at $o$ or $o'$. Similarly, we denote by $\tilde{\Gamma}^{(\ell)}(v)=(\tilde{V}_i^{(\ell)})_{i\le n_\ell(v)}$ the ancestral path of $v$ in $\tilde{\mathcal{G}}^{(\ell)}$. In this case, the path is necessarily finite and $\tilde{V}^{(\ell)}_{n_\ell(v)}\in \{o,o'\}$ defines the color of the vertex $v$.
\medskip

Let $(a_0,a_1,\ldots,a_d)$ denote the vertices on the path $[o,o']$, ordered by increasing distance from $o$. Thus $a_0=o$, $a_d=o'$, and $d(o, a_i) = i$.
Let
$$
\overline{\T}_i := \T_{a_i\setminus\{o,o'\}}
$$
be the subtree of $\T$ composed of $\{a_i\}$ and all the connected components of $\T\setminus\{a_i\}$ that contain neither $o$ nor $o'$
(see Figure \ref{fig:2seed} for an illustration of $\overline{\T}_i$). The family $(\overline{\T}_i)_{0\le i\le d}$ forms a partition of $\T$. Moreover, a vertex $v\in\overline{\T}_i$ is closer to $o$ than to $o'$ if $i<d/2$, at equal distance from both if $i=d/2$, and closer to $o'$ if $i>d/2$.

\begin{figure}
    \centering
\includegraphics[width=11cm]{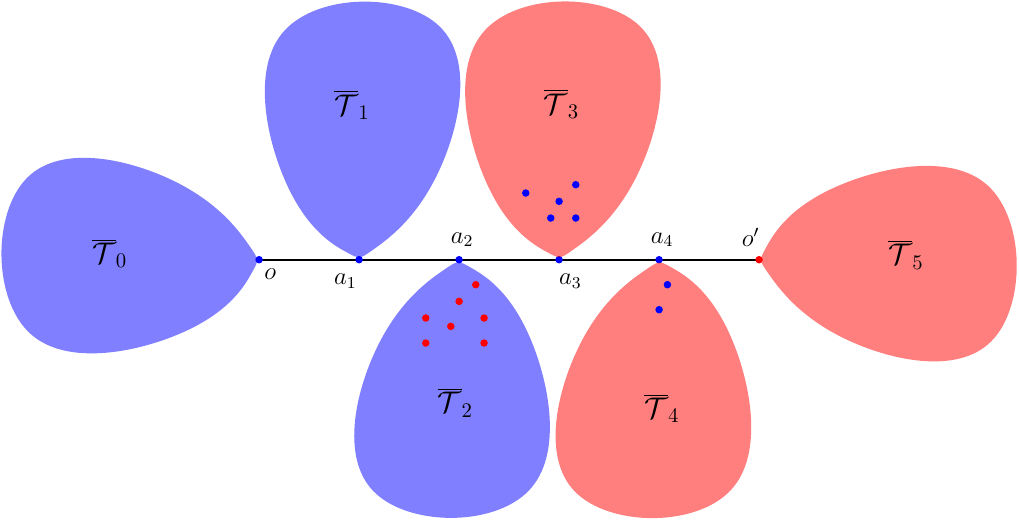}
\caption{\small{Illustration of the trees $(\overline{\T}_i)_{0\le i\le d}$ for $d=5$. The two-seed coloring process on $\T^{(\ell)}$ converges locally to the coloring shown in the figure. All vertices of $\overline{\T}_0$ (resp.\ $\overline{\T}_5$) are almost surely blue (resp.\ red). Moreover, $\overline{\T}_1$ and $\overline{\T}_2$ contain only finitely many red vertices, whereas $\overline{\T}_3$ and $\overline{\T}_4$ contain only finitely many blue vertices.}}
\label{fig:2seed}
\end{figure}

\begin{lemma}\label{lemm:derniersaut}
    Let $v\in \T$. We have
\begin{itemize}
    \item[(i)]  If $\tilde{\Gamma}(v)$ is finite, then for $\ell$ large enough, 
    $$\tilde{\Gamma}^{(\ell)}(v)=\tilde{\Gamma}(v).$$
    \item[(ii)]  If $\tilde{\Gamma}(v)$ is infinite, then, for any vertex $w\in \dir \tilde{\Gamma}(v)$ and for all $\ell$ large enough, the penultimate vertex of $\tilde{\Gamma}^{(\ell)}(v)$ is in $\T_w$, \emph{i.e.},
    $$\tilde{V}^{(\ell)}_{n_\ell(v)-1}\in \T_w.$$
\end{itemize}
\end{lemma}

\begin{proof}
Item (i) is a direct consequence of the previous proposition, applied to a ball containing the finite path $\tilde{\Gamma}(v)$. Assume now that $\tilde{\Gamma}(v)$ is infinite, and fix a vertex $w$ on the ray $\dir\tilde{\Gamma}(v)$. If $w=o$, there is nothing to prove. We therefore assume that $w\neq o$. There exists $N\ge 1$ such that
$\tilde{V}_n\in\T_w$ for every $n\ge N$. Set
$$
A_w:=B(w,d(w,o))\setminus(\T_w\cup\{o,o'\}).
$$
Since $A_w$ is finite and $\tau_{\tilde{V}_n}\to 0$, by possibly increasing $N$ we may assume that $\tau_{\tilde{V}_N}<\min\{\tau_u:\ u\in A_w\}$.

By the previous proposition, applied to a ball containing $(\tilde{V}_1,\ldots,\tilde{V}_N)$, the first $N$ vertices of $\tilde{\Gamma}^{(\ell)}(v)$ coincide with those of $\tilde{\Gamma}(v)$ for all sufficiently large $\ell$.
Let $Y^{(\ell)}$ denote the last non-seed vertex of the path $\tilde{\Gamma}^{(\ell)}(v)$ that belongs to $\T_w$. Since $\tilde{V}_N \in \T_w$, it implies that $Y^{(\ell)}$ appears at or after that vertex in the ancestral path:
$$
\tau_{Y^{(\ell)}}
\le \tau_{\tilde{V}_{N}}
<
\min\{\tau_u:\ u\in A_w\}.
$$
Consequently, the direct ancestor of $Y^{(\ell)}$ in $\tilde{\mathcal{G}}^{(\ell)}$ cannot belong to $A_w$. If this ancestor were not a seed, then, by the definition of $Y^{(\ell)}$, it could not belong to $\T_w$ either. It would then belong to $\T\setminus(\T_w\cup A_w\cup\{o,o'\})$, and hence would be farther from $Y^{(\ell)}$ than the seed $o$, which is impossible by the definition of the direct ancestor. Thus this direct ancestor belongs to $\{o,o'\}$, and hence $Y^{(\ell)}$ is the penultimate vertex of $\tilde{\Gamma}^{(\ell)}(v)$, \emph{i.e.}
$$\tilde{V}^{(\ell)}_{n_\ell(v)-1}\in \T_w.$$
\end{proof}

\begin{proof}[Proof of Proposition \ref{prop:twoseed}]
We assume here that $d=d(o,o')$ is odd and show almost-sure convergence of the coloring under the natural coupling. Furthermore, we establish that the limiting color of a vertex can be read from the ancestral path in $\tilde{\mathcal G}$: either from its terminal seed, if it is finite, or from the ``side'' of its asymptotic direction, if it is infinite. 
This follows directly from the following observations. Fix $v\in \T$; we have the following alternatives:
\begin{itemize}
    \item If $\tilde{\Gamma}(v)$ is finite and ends at $o$ (resp.\ $o'$), then the same property holds for $\tilde{\Gamma}^{(\ell)}(v)$ for $\ell$ large enough and so $v\in \tilde{B}^{(\ell)}$ (resp.\ $v\in \tilde{R}^{(\ell)}$) for $\ell$ large enough.
    \item Assume now that $\tilde{\Gamma}(v)$ is infinite, with $\dir\tilde{\Gamma}(v) \in \partial \overline{\T}_i$ and $i<d/2$. Choose a vertex $w$ on the ray $\dir\tilde{\Gamma}(v)$ such that $\T_w\subset \overline{\T}_i$. By the previous lemma, for $\ell$ large enough, $\tilde{V}^{(\ell)}_{n_\ell(v)-1}\in \T_w\subset \overline{\T}_i$. Thus, we have $d(\tilde{V}^{(\ell)}_{n_\ell(v)-1},o)<d(\tilde{V}^{(\ell)}_{n_\ell(v)-1},o')$ and so necessarily $\tilde{V}^{(\ell)}_{n_\ell(v)}=o$, \emph{i.e.}, $v\in \tilde{B}^{(\ell)}$.
    \item For a similar reason, $v\in \tilde{R}^{(\ell)}$ for $\ell$ large enough if $\dir\tilde{\Gamma}(v) \in \partial \overline{\T}_i$ with $i>d/2$.
\end{itemize}
These facts show the a.s. convergence stated above and the description of the limiting blue and red sets $\tilde{B}^{(\infty)}, \tilde{R}^{(\infty)}$ as
$$\tilde{B}^{(\infty)}:=\{v\in \T:\; \tilde{\Gamma}(v) \hbox{ ends at $o$  or } \tilde{\Gamma}(v) \hbox{ is infinite and }\dir \tilde{\Gamma}(v)\in \cup_{i<d/2} \partial\overline{\T}_i\}, $$
$$\tilde{R}^{(\infty)}:=\{v\in \T:\; \tilde{\Gamma}(v) \hbox{ ends at $o'$  or } \tilde{\Gamma}(v) \hbox{ is infinite and }\dir \tilde{\Gamma}(v)\in \cup_{i>d/2} \partial\overline{\T}_i\}. $$
To complete the proof of Proposition \ref{prop:twoseed}, it remains to show that 
 $$\tilde{B}^{(\infty)} \cap \{v \in \T : d(v,o)>d(v,o')\}\hbox{ is a.s. finite. }$$ 
 \emph{i.e.}, only finitely many vertices of $\overline{\T}_i$ for $i>d/2$ have an ancestral path $\tilde{\Gamma}(v)$ which ends at $o$ or is infinite with asymptotic direction $\dir \tilde{\Gamma}(v)\in \partial \overline{\T}_j$ for some $j<d/2$. Let $v$ be such a vertex. Then, there exists $z\in \tilde{\Gamma}(v)$ such that 
 $$z\in \cup_{i>d/2} \overline{\T}_i \qquad \hbox{ and } \qquad \tilde{p}(z) \in  \cup_{i<d/2}\overline{\T}_i.$$
 Using the fact that $d(\tilde{p}(z),z)\le d(z,o')$, we get that $\tilde{p}(z)\in B(o',d-1)\setminus\{o'\}$.
 Thus
 $$\tilde{B}^{(\infty)} \cap \{v \in \T : d(v,o)>d(v,o')\}\subset \bigcup_{u\in B(o',d-1)\setminus\{o'\}}\tilde{D}(u),$$
 where we recall that $\tilde{D}(u)$ denotes the set of descendants of $u$ in the graph $\tilde{\mathcal{G}}$. We conclude using the fact that $B(o',d-1)\setminus\{o'\}$ is a finite set, together with Theorem \ref{thm:GwithSeeds}, which states that $\tilde{D}(u)$ is a.s. finite for $u\notin \{o,o'\}$. The second finiteness statement follows by symmetry.
\end{proof}
\medskip

\begin{remark}
When $d(o,o')$ is even, the natural coupling does not give an a.s. local limit in general. Indeed, for a vertex whose infinite ancestral path has asymptotic direction in $\partial\overline{\T}_{d/2}$, the penultimate vertex of $\tilde{\Gamma}^{(\ell)}(v)$ eventually lies in the middle subtree, where the two seeds are at the same distance. As $\ell$ grows, infinitely many fresh tie-breakings between $o$ and $o'$ are therefore used, and the color of such a vertex does not stabilize, but takes both colors infinitely often. This is why Theorem \ref{theo:twoseed} proved below is stated as a convergence in distribution, with an additional Bernoulli choice for each middle-direction component.
\end{remark}
\medskip

\begin{proof}[Proof of Theorem \ref{theo:twoseed}]
It suffices to show that for any finite collection of vertices $(v_1,\ldots,v_k)$, the following properties hold:
\begin{itemize}
    \item[(i)] If $\tilde{\Gamma}(v_i)$ is finite and ends at $o$ (resp.\ $o'$), or if $\dir \tilde{\Gamma}(v_i)\in \partial \overline{\T}_j$ with $j<d/2$ (resp.\ $j>d/2$), then $v_i\in \tilde{B}^{(\ell)}$ (resp.\ $\tilde{R}^{(\ell)}$) for all sufficiently large $\ell$.

    \item[(ii)] If $\dir \tilde{\Gamma}(v_i)\in \partial \overline{\T}_{d/2}$, define
    $$
    \mathcal I:=\bigl\{i'\le k:\ \tilde{\Gamma}(v_{i'}) \text{ coalesces with } \tilde{\Gamma}(v_i)\bigr\}.
    $$
    Then, for all sufficiently large $\ell$, the vertices $(v_{i'})_{i'\in\mathcal I}$ receive the same color in the two-seed coloring process on $\T^{(\ell)}$, and this color is blue or red with probability $1/2$ each. Moreover, this color is asymptotically independent of the colors of $(v_{i'})_{i'\notin\mathcal I}$.
\end{itemize}

Property (i) was established in Proposition \ref{prop:twoseed} when $d=d(o,o')$ is odd. The proof given there in fact applies verbatim when $d$ is even and $j \neq d/2$. We now prove (ii). We assume that $d$ is even, since otherwise there is nothing to prove. Also, to avoid cumbersome notation, we only treat the case of two vertices $v_1$ and $v_2$; the general case follows by the same argument. 
\medskip

Assume that $\dir \tilde{\Gamma}(v_1)\in \partial \overline{\T}_{d/2}$. Suppose first that $\tilde{\Gamma}(v_1)$ and $\tilde{\Gamma}(v_2)$ coalesce. Then, for all sufficiently large $\ell$, the truncated ancestral paths $\tilde{\Gamma}^{(\ell)}(v_1)$ and $\tilde{\Gamma}^{(\ell)}(v_2)$ also coalesce. Consequently, $v_1$ and $v_2$ receive the same color in the two-seed coloring process on $\T^{(\ell)}$. By Lemma \ref{lemm:derniersaut}, the penultimate vertex of $\tilde{\Gamma}^{(\ell)}(v_1)$ is in $\overline{\T}_{d/2}$ for all sufficiently large $\ell$. Since this vertex is at equal distance from $o$ and $o'$, it is connected to either seed with probability $1/2$ each. This proves (ii) in the case of coalescing paths.

\medskip

Assume now that $\tilde{\Gamma}(v_1)$ and $\tilde{\Gamma}(v_2)$ do not coalesce. We show that, for all sufficiently large $\ell$, the paths $\tilde{\Gamma}^{(\ell)}(v_1)$ and $\tilde{\Gamma}^{(\ell)}(v_2)$ are disjoint, except possibly for their last vertex. By Item 5 of Theorem \ref{thm:GwithoutSeeds}, we have
$$
\dir \tilde{\Gamma}(v_1)\neq \dir \tilde{\Gamma}(v_2).
$$
(We assume here that $\tilde{\Gamma}(v_2)$ is infinite; the finite case is even simpler.) Hence, there exist vertices $w_1\in \overline{\T}_{d/2}$ and $w_2\in \T$ such that $\T_{w_1}$ and $\T_{w_2}$ are disjoint, while
$$
\dir \tilde{\Gamma}(v_1)\in \partial \T_{w_1},
\qquad
\dir \tilde{\Gamma}(v_2)\in \partial \T_{w_2}.
$$
By Lemma \ref{lemm:derniersaut}, for all sufficiently large $\ell$,
$$
\tilde{V}^{(\ell),1}_{n_\ell(v_1)-1}\in {\T}_{w_1},
\qquad
\tilde{V}^{(\ell),2}_{n_\ell(v_2)-1}\in {\T}_{w_2},
$$
where $(\tilde{V}_n^{(\ell),i})_{n\ge1}$ denotes the ancestral path of $v_i$ in $\T^{(\ell)}$. Since $\T_{w_1}$ and $\T_{w_2}$ are disjoint, the paths $\tilde{\Gamma}^{(\ell)}(v_1)$ and $\tilde{\Gamma}^{(\ell)}(v_2)$ are disjoint, except possibly at their final vertex. As before, the fact that $\dir \tilde{\Gamma}(v_1)\in \partial \overline{\T}_{d/2}$ implies that the penultimate vertex of $\tilde{\Gamma}^{(\ell)}(v_1)$ is at equal distance from $o$ and $o'$. Hence, the color of $v_1$ is blue or red with probability $1/2$ each, depending on the increasing ordering attached to $\tilde{V}^{(\ell),1}_{n_\ell(v_1)-1}$. Since this ordering is independent of the ancestral path of $v_2$ (hence of the color of $v_2$), it follows that
$$
\lim_{\ell \to \infty}\P\bigl(v_1\in \tilde{B}^{(\ell)} \mid v_2\in \tilde{B}^{(\ell)}\bigr)=\lim_{\ell \to \infty}\P\bigl(v_1\in \tilde{B}^{(\ell)} \mid v_2\in \tilde{R}^{(\ell)}\bigr)
=
\frac12.
$$
This proves (ii) and completes the proof of the theorem.
\end{proof}

\subsection{Local limit of a tree with seeds at the root and the leaves (Theorem \ref{theo:main})}

We now consider the coloring process on the finite tree $\T^{(\ell)}$, with a blue seed at the root $o$ and a red seed at each leaf of $\T^{(\ell)}$. As usual, we let $B^{(\ell)}$ and $R^{(\ell)}$ denote the final sets of blue and red vertices, respectively.
We also consider the genealogical graph $\tilde{\mathcal{G}}^{o}$ on $\T$ with a single seed located at $o$. These processes are constructed on the same probability space using the same arrival times $(\tau_w)_{w\in \T}$ and the same random orderings $((\zi_j^w)_{j\ge 1})_{w\in \T}$ used to break ties.
\medskip

As already observed in the introduction, if a vertex $v$ is colored blue in $\T^{(\ell)}$, then it is also colored blue in all subsequent $\T^{(\ell')}$ for every $\ell'>\ell$. Hence,
$$
B^{(\ell)}\subset B^{(\ell')} \qquad \text{for every $\ell<\ell'$}
$$
and, as $\ell \to \infty$, the set $B^{(\ell)}$ converges a.s. to $\tilde{\mathcal{U}}_o$, the set of vertices of $\T$ whose ancestral path ends at the root $o$, \emph{i.e.}, the genealogical component of $\tilde{\mathcal{G}}^{o}$ containing the root. This proves the local-limit statement in Theorem \ref{theo:main}. Furthermore, recalling the notation
$$
m_\ell:=\min\{|v|:\ v\in R^{(\ell)}\}
$$
for the minimum height of the red vertices, we find that
\begin{equation*}
m_\infty := \lim_{\ell \to \infty} m_{\ell}
=
\min\{|v|:\ v\notin \tilde{\mathcal{U}}_o\}.
\end{equation*}
Now, by Theorem \ref{thm:GwithSeeds}, the random variable $m_\infty$ is almost surely finite. Moreover, Proposition \ref{prop:keyresult} implies that
$$
\P(m_\infty=1)>0.
$$
This proves statement \eqref{eq:min} of Theorem \ref{theo:main}.
\medskip

We now turn to the study of
$$
M_\ell:=\max\{|v|:\ v\in B^{(\ell)}\}.
$$
Again, we have
$$
M_\infty := \lim_{\ell\to\infty} M_\ell = \max\{|v|:\ v\in \tilde{\mathcal{U}}_o\},
$$
and, by Item 2(b) of Theorem \ref{thm:GwithSeeds}, the component $\tilde{\mathcal{U}}_o$ is almost surely infinite. Consequently, $M_\infty=\infty$ a.s. This completes the part of Theorem \ref{theo:main} which follows from local convergence and from the infinite-volume genealogical picture.
It remains only to prove the last part of Theorem \ref{theo:main}, namely that, with high probability, the blue vertices reach the leaves of $\T^{(\ell)}$ as $\ell\to\infty$, \emph{i.e.},
\begin{equation}\label{eq:touchleaves}
\lim_{\ell\to\infty}\P(M_\ell=\ell-1)=1,
\end{equation}
which is exactly the estimate \eqref{eq:max}, restated here for convenience. This estimate is of a different nature: it concerns the behavior of the coloring near the boundary of $\T^{(\ell)}$, which escapes to infinity, and therefore cannot be read from the convergence of any fixed neighborhood of the root.

\section{The blue vertices reach the leaves}\label{sec:max}

This last section is devoted solely to the proof of \eqref{eq:touchleaves}, and hence completes the proof of Theorem \ref{theo:main}. It is somewhat standalone: rather than reconstructing colors through backward genealogical explorations, as in the previous sections, we follow the evolution of the coloring process forward in time. 
\medskip 

The proof uses a recursive argument on finite descendant subtrees and a coupling with a supercritical Galton--Watson process. For this reason, it is convenient to work with a generic finite rooted $K$-ary tree $\Tf$ instead of $\T^{(\ell)}$. The only discrepancy between the two trees is at the global root, which has $K+1$ children in $\T^{(\ell)}$ and only $K$ children in $\Tf$; this first-generation difference will play no role, so we will work with $\Tf$ from now on to avoid having to distinguish the case of the root vertex.
\medskip

\noindent We use the following notation throughout the proof.
\begin{itemize}
\item $\Tf$ denotes a finite rooted $K$-ary tree, with root $o$ and height $|\Tf|$. We write $\partial \Tf$ for its set of leaves, and $|v|$ for the height of a vertex $v\in \Tf$.

\item For a vertex $v\in \Tf$, we denote by $\Tf_v$ the subtree of $\Tf$ rooted at $v$, consisting of all its descendants. Thus $\Tf_v$ is itself a rooted $K$-ary tree of height $|\Tf|-|v|$. Whenever this notation is applied to a rooted subtree, heights are understood relative to the root of that subtree.

\item For two disjoint subsets $B_0,R_0\subset \Tf$, we consider the coloring process on $\Tf$ with blue seeds in $B_0$ and red seeds in $R_0$. The vertices in $B_0\cup R_0$ have arrival time $0$, while the variables $(\tau_u)_{u\notin B_0\cup R_0}$ are i.i.d.\ exponential random variables with parameter $1$. For $t \ge 0$, we denote by $B_t$ (resp.\ $R_t$) the set of vertices colored blue (resp.\ red) up to time $t$.

\item Finally, $X^{\Tf}_1,X^{\Tf}_2,\ldots$ denote the vertices of $\Tf\setminus\{o\}$ whose heights are strictly smaller than $|\Tf|/2$, ordered according to their coloring times in increasing order.
\end{itemize}
Note that, once $v$ has been colored, it separates the subtree $\Tf_v$ from the rest of the tree: conditionally on the state $(B_{\tau_v}\cap \Tf_v,R_{\tau_v}\cap \Tf_v)$, the evolution of the coloring inside $\Tf_v$ after time $\tau_v$ depends only on the clocks $(\tau_w)_{w\in \Tf_v}$ and on the increasing orderings $((\zi_j^w)_{j\ge 1})_{w\in \Tf_v}$, and is therefore independent of the exterior.

\medskip

The event $\mathcal{A}_i$ defined below plays the role of a ``good event'' for the recursive step: it says that $X_i^{\Tf}$ becomes a blue vertex near the middle of the tree, while the part of its descendant subtree that has already been colored remains confined to a bounded layer near the leaves, which enables the recursion to be applied in $\Tf_{X_i^{\Tf}}$. More precisely, for $C,H\ge 0$ and $i\ge 1$, we define the event $\mathcal{A}_i:=\mathcal{A}_i(\Tf,C,H,B_0,R_0)$ by
\begin{equation}\label{eq:defai}
\begin{aligned}
\mathcal{A}_i := \Bigl\{&
X^{\Tf}_i \text{ is colored blue},\\
& |\Tf|/2 - H < |X^{\Tf}_i| < |\Tf|/2,\\
& \text{at time }(\tau_{X^{\Tf}_i})^{-},\ 
\Tf_{X^{\Tf}_i} \text{ contains at most } C \text{ colored non-leaf vertices,}\\
& \text{and all these vertices are at distance at most } 2H \text{ from the leaves}
\Bigr\}.
\end{aligned}
\end{equation}

\subsection{Coupling with a supercritical Galton--Watson process}

The next proposition is the key ingredient in the proof. It states that with high probability one can find two disjoint descendant subtrees in which the good event $\mathcal{A}_i$ occurs. Thanks to the independence observation mentioned previously, the coloring can then be restarted in these subtrees, which will allow us to compare the growth of the blue region with a supercritical Galton--Watson process.

\begin{prop}\label{prop:main}
Recall the definition of the events $\mathcal{A}_i$ given in \eqref{eq:defai}.
Let $\eta>0$. There exist $H,C,n_0\ge 0$ such that, for all $|\Tf|\ge n_0$, and for any choice of vertices $w_1,\ldots,w_C$ of height at least $|\Tf|-2H$, starting from the initial blue seed $B_0=\{o\}$ and red seeds $R_0=\partial \Tf \cup\{w_1,\ldots,w_C\}$, we have
\begin{equation}\label{eq1}
\P\bigl(\exists i<j,\ \Tf_{X^{\Tf}_i}\cap \Tf_{X^{\Tf}_j}=\emptyset \ \text{ and }\ \mathcal{A}_i\cap \mathcal{A}_j\bigr)\ge 1-\eta.
\end{equation}
\end{prop}

\noindent See Figure \ref{fig:Ai} for an illustration of the initial condition and of the event $\mathcal{A}_1$.

\begin{figure}
    \centering
\includegraphics[width=11cm]{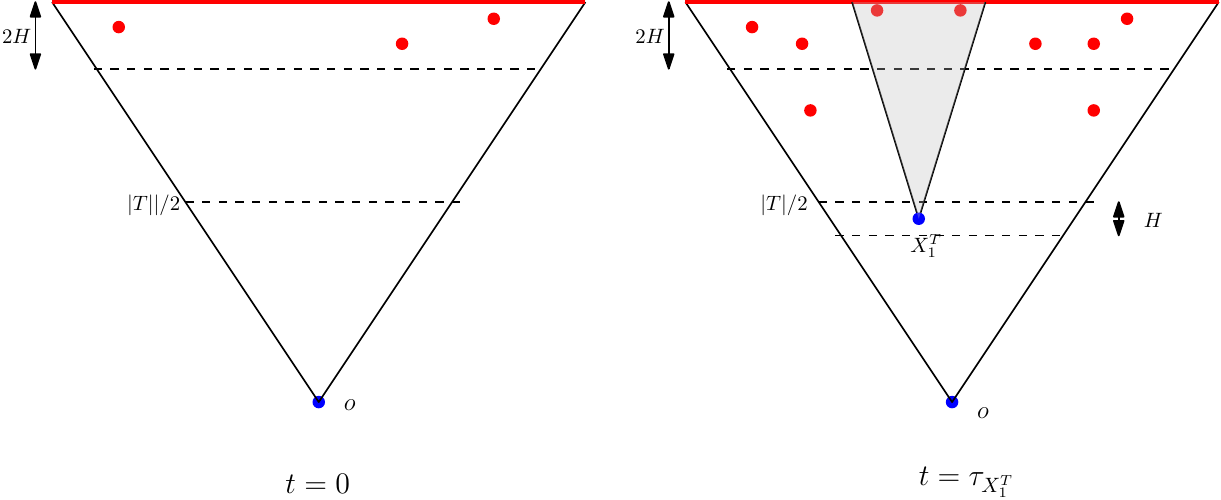}
\caption{\small{Illustration of the initial condition in Proposition \ref{prop:main} and of the event $\mathcal{A}_1$, with $C=3$. If $\mathcal{A}_1$ occurs, then at time $t=\tau_{X_1^{\Tf}}$, at most $C$ colored non-leaf vertices lie in the subtree $\Tf_{X_1^{\Tf}}$, and they are all close (at distance at most $2 H$) from its leaves.}}
\label{fig:Ai}
\end{figure}
  
\medskip

We first explain how Proposition \ref{prop:main} implies \eqref{eq:touchleaves} and hence completes the proof of Theorem \ref{theo:main}, before proving the proposition itself. 

\begin{proof}[Proof of \eqref{eq:touchleaves} using Proposition \ref{prop:main}]
Set
$$
\mathcal{C}_{i,j}
:=
\mathcal{C}_{i,j}(\Tf,C,H,B_0,R_0)
:=
\bigl\{\Tf_{X^{\Tf}_i}\cap \Tf_{X^{\Tf}_j}=\emptyset \ \text{and}\ \mathcal{A}_i\cap \mathcal{A}_j\bigr\}.
$$
Proposition \ref{prop:main} asserts that, for suitable choices of $(\Tf,C,H,B_0,R_0)$, with high probability there exist indices $i<j$ such that $\mathcal{C}_{i,j}$ occurs. 
\medskip

Fix $\eta\in(0,1/4)$ and choose $H,C,n_0\ge 0$ such that \eqref{eq1} holds. By possibly increasing $n_0$, we may assume that $C<K^{n_0-2H-1}$ and $n_0/2-2H>2H$. Consider the following process $(\mathcal{V}_l)_{l\ge 0}$ taking values in the subsets of $\Tf$:
\begin{itemize}
\item $\mathcal{V}_0=\{o\}$. If there exist $i<j$ such that $\mathcal{C}_{i,j}(\Tf,C,H,\{o\},\partial \Tf)$ occurs, we set $\mathcal{V}_1:=\{X^{\Tf}_i,X^{\Tf}_j\}$ and call $X^{\Tf}_i,X^{\Tf}_j$ the children of $o$. Otherwise, we set $\mathcal{V}_1=\emptyset$.

By definition of the event $\mathcal{A}_i$, for any vertex $v\in \mathcal{V}_1$, at time $(\tau_v)^-$, the subtree $\Tf_v$ contains at most $C$ colored vertices that are not leaves and have height at least $|\Tf_v|-2H$. Note that these vertices must be red, since they are at distance at most $2H$ from the leaves and at distance at least $|\Tf|/2-2H\ge n_0-2H>2H$ from any blue vertex. Hence $B_{\tau_v}\cap \Tf_v=\{v\}$.
Moreover, for $v_1\neq v_2$ in $\mathcal{V}_1$, we have $\Tf_{v_1}\cap \Tf_{v_2}=\emptyset$. Thus, conditionally on $(B_{\tau_v},R_{\tau_v})_{v\in \mathcal{V}_1}$, the processes
$\bigl((B_t\cap \Tf_v)_{t\ge \tau_v}, (R_t\cap \Tf_v)_{t\ge \tau_v}\bigr), \quad v\in \mathcal{V}_1$
are independent.

\item Assume that $\mathcal{V}_l=\{U_1^{(l)},\ldots,U_{\sharp\mathcal{V}_l}^{(l)}\}$. For each $v\in \mathcal{V}_l$, consider the coloring process on $\Tf_v$ for $t\ge \tau_v$. By induction, these processes are independent for $v\in \mathcal{V}_l$. Each of them starts from an initial configuration where only the root $v$ is blue, while the red vertices are the leaves together with at most $C$ additional vertices, each located at distance at most $2H$ from the leaves.

If there exist $i<j$ such that $\mathcal{C}_{i,j}(\Tf_v,C,H,\{v\},R_{\tau_v}\cap \Tf_v)$ occurs, we declare $\{X^{\Tf_v}_i,X^{\Tf_v}_j\}$ to be the children of $v$. We then define $\mathcal{V}_{l+1}$ as the union of the children of all vertices $v\in \mathcal{V}_l$.
\end{itemize}

By induction, we have that if $v\in \mathcal{V}_l$, then
$$
\frac{|\Tf|}{2^l}\le |\Tf_v|\le \frac{|\Tf|}{2^l}+2H.
$$
Let $L\in \mathbb{N}$ be such that
$
\frac{|\Tf|}{2^L}\ge n_0 > \frac{|\Tf|}{2^{L+1}}
$
that is,
$$
L:=\left\lfloor \frac{\log(|\Tf|/n_0)}{\log 2} \right\rfloor.
$$
Thus, for every $v\in \mathcal{V}_L$, we have
$$
n_0 \le |\Tf_v| \le 2n_0 + 2H.
$$

Using Proposition \ref{prop:main}, we see that $(\sharp\mathcal{V}_l)_{l\le L}$ stochastically dominates a Galton--Watson process $(Z_l)_{l\le L}$ starting from $Z_0=1$ with offspring distribution
$$
\mu(0)=\eta, \qquad \mu(2)=1-\eta.
$$
Since $\eta<1/4$, this Galton--Watson process is supercritical. It survives with probability $(1-2\eta)/(1-\eta)$, and for $l$ large enough,
$$
\P\bigl(Z_l>(2(1-2\eta))^l \mid Z_l>0\bigr)\ge 1-\eta.
$$
Hence, if $|\Tf|$ is large enough, we obtain
$$
\P\bigl(\sharp\mathcal{V}_L\ge (2(1-2\eta))^L\bigr)
\ge \P\bigl(Z_L>(2(1-2\eta))^L\bigr)
\ge 1-2\eta.
$$
Set $\beta:=\frac{\log(2(1-2\eta))}{\log 2} >0$ so that
$$
(2(1-2\eta))^L = 2^{\beta L}
\ge \left(2^{\frac{\log(|\Tf|/n_0)}{\log 2}-1}\right)^\beta
= \left(\frac{|\Tf|}{2n_0}\right)^\beta.
$$
Therefore,
$$
\P\left(\sharp\mathcal{V}_L\ge \left(\frac{|\Tf|}{2 n_0}\right)^\beta\right)\ge 1-2\eta.
$$

This shows that, with probability at least $1-2\eta$, there exist at least $(|\Tf|/2n_0)^\beta$ disjoint subtrees of height between $n_0$ and $2n_0+2H$ such that, in each of them, the coloring process colors the root $v$ blue and, at the time when it becomes blue, the subtree contains at most $C$ red vertices that are not leaves and have height at least $|\Tf_v|-2H$. Each of these subtrees then evolves independently.

Using the facts that $C<K^{n_0-2H-1}$ and $|\Tf_v|\ge n_0$, we obtain that the number of vertices in $\Tf_v$ at height $|\Tf_v|-2H$ is strictly larger than $C$. Therefore, at time $\tau_v$, there necessarily exists a path from the root $v$ to a leaf consisting entirely of uncolored vertices (except, of course, for the root and the leaf). Hence, the probability that no blue vertex reaches distance one from the leaves is strictly less than one. Since the sizes of these $(|\Tf|/2n_0)^\beta$ trees are uniformly bounded, we in fact obtain a uniform upper bound $\alpha(n_0,H,C)<1$ for the probability that a given tree has no blue vertex at distance one from the leaves.

Thus, writing $M_{\Tf}$ for the maximal height reached by a blue vertex in the coloring process on $\Tf$, we obtain
$$
\P(M_{\Tf}<|\Tf|-1)\le 2\eta+ \alpha(n_0,H,C)^{(|\Tf|/2n_0)^\beta},
$$
which can be made arbitrarily small as $|\Tf|\to\infty$, provided that $\eta$ is chosen sufficiently small.

\end{proof}

\begin{remark} 
    The proof that the blue vertices reach the leaves presented here is rather intricate, and the definition of the event $\mathcal{A}_i$ used to run the recursion may seem overcomplicated at first, but this complexity is necessary. For instance, requiring that the whole subtree $\Tf_{X_i^{\Tf}}$ be uncolored at time $(\tau_{X_i^{\Tf}})^-$ would be too restrictive and prevent the recursion from being applied. Indeed, after one step of the construction, the descendant subtrees naturally inherit a bounded number of red vertices close to their leaves. The event $\mathcal{A}_i$ allows precisely this boundary layer and keeps the class of admissible initial configurations stable under the recursion.
\end{remark}

\subsection{Existence of a blue vertex at height around $|\Tf|/2$}

In order to prove Proposition \ref{prop:main}, we first establish a weaker version by showing that there exists $N$ such that, with high probability, one of the first $N$ non-root vertices of height smaller than $|\Tf|/2$ to be colored is colored blue and has an almost uncolored subtree in front of it. This is the content of Proposition \ref{prop:unsommetvide} below, which will later be strengthened to Proposition \ref{prop:main}.

\begin{prop}\label{prop:unsommetvide} Let $\eta>0$. Then, there exist $H,C,N,n_0\ge 0$ such that, if $|\Tf|\ge n_0$, then for any choice of vertices $w_1,\ldots,w_C$ of height greater than or equal to $|\Tf|-2H$, starting from $B_0=\{o\}$ and $R_0=\partial \Tf \cup\{w_1,\ldots,w_C\}$, we have
\begin{equation*}
 \P(\forall 1\le i\le N,\; \mathcal{A}_i^c)\le \eta.
\end{equation*}
\end{prop}

\begin{proof}[Proof of Proposition \ref{prop:unsommetvide}]

For $v\in \Tf$ and $r,h\ge 0$, we denote
$$
B(v,r):=\{w\in \Tf:\ d(v,w)\le r\}
\qquad \text{and} \qquad
B_{\ge h}(v,r):=\{w\in \Tf:\ d(v,w)\le r \ \text{and}\ |w|\ge h\},
$$
the ball centered at $v$ of radius $r$ and the set of vertices in this ball whose height is at least $h$, respectively.
For $v_1,v_2\in \Tf$, we denote by $v_1\wedge v_2$ their most recent common ancestor.
Recall that $R_0=\partial \Tf \cup\{w_1,\ldots,w_C\}$ denotes the set of initial red vertices, and that $X^{\Tf}_1,X^{\Tf}_2,\ldots$ enumerates the vertices of $\Tf\setminus\{o\}$ whose heights are strictly smaller than $|\Tf|/2$, ordered according to their coloring times. Since $\Tf$ is fixed throughout this section, we will simplify the notation and write $X_1,X_2,\ldots$ instead of $X^{\Tf}_1,X^{\Tf}_2,\ldots$.
We now introduce the following events.

\begin{eqnarray*}
\mathcal{E}_1=\mathcal{E}_1(\Tf,N,H,C,R_0) &:=& \{\forall i\neq j \le N, |X_i\wedge X_j| < |\Tf|/8\},\\
\mathcal{E}_2=\mathcal{E}_2(\Tf,N,H,C,R_0) &:=& \{\forall i \le N, |X_i|> |\Tf|/2-H\},\\
\mathcal{E}_3=\mathcal{E}_3(\Tf,N,H,C,R_0) &:=& \{\forall i \le N, \forall k \le C, X_i\notin B(w_k,|\Tf|/2+H)\},\\
\mathcal{E}_4=\mathcal{E}_4(\Tf,N,H,C,R_0) &:=&\{\forall i \le N, \mathcal{A}_i^c\}\cap  \mathcal{E}_1 \cap  \mathcal{E}_2 \cap  \mathcal{E}_3.
\end{eqnarray*}
Thus, we have
$$
\P\bigl(\forall i \le N,\; \mathcal{A}_i^c\bigr)
\le \P(\mathcal{E}_1^c)+\P(\mathcal{E}_2^c)+\P(\mathcal{E}_3^c)+\P(\mathcal{E}_4).
$$
The first three terms on the right-hand side are easy to bound.

Indeed, to bound $\P(\mathcal{E}_1^c)$, observe that $X_i$ and $X_j$ are distributed as two distinct vertices chosen uniformly at random from the complete $K$-ary tree of height $\lceil |\Tf|/2-1 \rceil$, excluding the root. It is easy to see that the height of the most recent common ancestor of such vertices is stochastically dominated by a geometric random variable with parameter $(K-1)/K$. Hence, we obtain
\begin{equation}\label{eq:E1}
\P(\mathcal{E}_1^c)
\le N^2 \P\bigl( |X_1 \wedge X_2| \ge |\Tf|/8 \bigr)
\le N^2 \frac{1}{K^{|\Tf|/8}}.
\end{equation}

Concerning $\P(\mathcal{E}_2^c)$, and still using that $X_i$ is uniformly distributed on the complete $K$-ary tree of height $\lceil |\Tf|/2-1 \rceil$, excluding the root, together with a simple counting argument on the number of vertices of height smaller than $|\Tf|/2-H$, we obtain
\begin{equation}\label{eq:E2}
\P(\mathcal{E}_2^c)\le N K^{-H}.
\end{equation}

For $\P(\mathcal{E}_3^c)$, recall that by assumption, for any $k\le C$, $|w_k|\ge |\Tf|-2H$, whereas for any $i\ge 1$ we have $|X_i|\le \lceil |\Tf|/2-1 \rceil$. Let $z_k$ denote the ancestor of $w_k$ at height $\lceil |\Tf|/2-1 \rceil$. Then
$$
X_i\in B(w_k,|\Tf|/2+H)\ \Longrightarrow\ X_i\in B(z_k,3H).
$$
Hence,
\begin{equation}\label{eq:E3}
\P(\mathcal{E}_3^c)
\le CN\,\P\bigl(X_1\in B(z_1,3H)\bigr)
\le CN \frac{K^{3H}}{K^{|\Tf|/2-2}},
\end{equation}
again using that $X_1$ is uniformly distributed on the complete $K$-ary tree of height $\lceil |\Tf|/2-1 \rceil$ (excluding the root).
\medskip

It remains to upper bound $\P(\mathcal{E}_4)$, which requires more work. In the following, for a vertex $v$ and a time $t>0$, we say that the clock of $v$ has rung before time $t$ if $0<\tau_v<t$. We introduce the following two events:
\begin{itemize}
\item[(i)] $\mathcal{F}_i := \{\text{At time } (\tau_{X_i})^{-}, \text{ strictly more than } C \text{ clocks have rung in } \Tf_{X_i}\}$.
\item[(ii)] $\mathcal{O}_i := \{\text{A clock has rung in } B_{\ge |\Tf|/2}(X_i, |X_i|) \text{ before time } (\tau_{X_i})^{-}\}$.
\end{itemize}

\begin{lemma} Assume that $H< |\Tf|/4$. Then
   we have $$\mathcal{E}_4\subset \{\forall i\le N, \mathcal{O}_i \cup \mathcal{F}_i\}\cap \mathcal{E}_1\cap \mathcal{E}_2.$$
\end{lemma}

\begin{proof}

Recalling the definition of $\mathcal{A}_i$ given in \eqref{eq:defai}, we have
$
\mathcal{A}_i := \mathcal{X}_i \cap \mathcal{Y}_i \cap \mathcal{Z}_i,
$
where
\begin{eqnarray*}
\mathcal{X}_i &:=& \{|X_i|>|\Tf|/2-H\},\\
\mathcal{Y}_i &:=& \{\text{$X_i$ is colored blue}\},\\
\mathcal{Z}_i &:=& \Bigl\{\text{At time } (\tau_{X_i})^{-},\; \Tf_{X_i} \text{ contains at most } C \text{ colored vertices that are not leaves,}\\
&& \qquad\quad \text{and each of them is at distance at most } 2H \text{ from the leaves}\Bigr\}.
\end{eqnarray*}
We prove that
$$
\mathcal{F}_i^c \cap \mathcal{O}_i^c \cap \mathcal{E}_1 \cap \mathcal{E}_2 \cap \mathcal{E}_3 \subset \mathcal{A}_i,
$$
by showing that the event on the left-hand side implies $\mathcal{X}_i$, $\mathcal{Y}_i$, and $\mathcal{Z}_i$.

\medskip

\noindent We have $\mathcal{E}_2 \subset \mathcal{X}_i$. Assume now that $\mathcal{O}_i^c \cap \mathcal{E}_1 \cap \mathcal{E}_2 \cap \mathcal{E}_3$ occurs. Observe that for $i \neq j \le N$,
$$
d(X_i,X_j)
= |X_i| + |X_j| - 2|X_i \wedge X_j|
> |X_i| + \frac{|\Tf|}{2} - H - \frac{|\Tf|}{4}
> |X_i|.
$$
Hence, we have $X_j \notin B(X_i,|X_i|)$. By definition of the sequence $(X_j)_{j\le N}$, it follows that at time $(\tau_{X_i})^{-}$ no clock in $B(X_i,|X_i|)$ at height smaller than $|\Tf|/2$ has rung. Moreover, since $\mathcal{O}_i^c$ holds, no clock in $B(X_i,|X_i|)$ at height at least $|\Tf|/2$ has rung either. Therefore, no clock in $B(X_i,|X_i|)$ has rung before time $(\tau_{X_i})^{-}$. Using now $\mathcal{E}_3$ and the fact that $|X_i|<|\Tf|/2$, we deduce that $B(X_i,|X_i|)$ contains no vertex of $R_0$. Hence, at time $(\tau_{X_i})^{-}$, the ball $B(X_i,|X_i|)$ contains no red vertex but does contain the root, which is blue. Consequently, $X_i$ becomes blue, and we obtain
$$
\mathcal{O}_i^c \cap \mathcal{E}_1 \cap \mathcal{E}_2 \cap \mathcal{E}_3 \subset \mathcal{Y}_i.
$$

Moreover, on the event $\mathcal{O}_i^c \cap \mathcal{E}_1 \cap \mathcal{E}_2$, we have $|X_i|>|\Tf|/2-H$ and no clock has rung in $B(X_i,|X_i|)$ before time $\tau_{X_i}$. In particular, any clock that has rung in $\Tf_{X_i}$ before time $\tau_{X_i}$ must correspond to a vertex at height at least $2|X_i|\ge |\Tf|-2H$.
Hence, by definition of $\mathcal{F}_i$, on the event
$\mathcal{F}_i^c \cap \mathcal{O}_i^c \cap \mathcal{E}_1 \cap \mathcal{E}_2$,
at most $C$ vertices have rung in $\Tf_{X_i}$ before time $\tau_{X_i}$, and all of them lie at distance at most $2H$ from the leaves. Using additionally $\mathcal{E}_3$, which ensures that the initial red vertices in $\Tf_{X_i}$ are only leaves, we obtain
$$
\mathcal{F}_i^c \cap \mathcal{O}_i^c \cap \mathcal{E}_1 \cap \mathcal{E}_2 \cap \mathcal{E}_3
\subset \mathcal{Z}_i.
$$
Combining the previous observations, we deduce that
$$
\mathcal{F}_i^c \cap \mathcal{O}_i^c \cap \mathcal{E}_1 \cap \mathcal{E}_2 \cap \mathcal{E}_3
\subset \mathcal{A}_i,
$$
which implies
$$
\mathcal{A}_i^c \cap \mathcal{E}_1 \cap \mathcal{E}_2 \cap \mathcal{E}_3
\subset \mathcal{F}_i \cup \mathcal{O}_i.
$$
Hence,
$$
\mathcal{E}_4 \subset \Bigl\{\forall i \le N,\; \mathcal{F}_i \cup \mathcal{O}_i\Bigr\} \cap \mathcal{E}_1 \cap \mathcal{E}_2.
$$
This concludes the proof of the lemma.
\end{proof}

Let us now derive an upper bound on the probability of
$\{\forall i \le N,\; \mathcal{F}_i \cup \mathcal{O}_i\}\cap \mathcal{E}_1\cap \mathcal{E}_2$.
Recall that $\tau_v$ denotes the coloring time of a vertex $v$. We define the $\sigma$-algebra
$$
\mathcal{H}
:= \sigma\bigl(\{\tau_v:\ |v|<|\Tf|/2\}\bigr),
$$
so that $\mathcal{E}_1,\mathcal{E}_2,\mathcal{E}_3 \in \mathcal{H}$. For $i\le N$, we denote by
$\tau_i$ the coloring time of $X_i$, i.e. $\tau_i=\tau_{X_i}$. Note that $\tau_i$ is $\mathcal{H}$-measurable.
We observe that for any $i\le N$ and any $u\in B_{\ge |\Tf|/2}(X_i,|X_i|)$, we have
$$
2|u\wedge X_i|
= |u| + |X_i| - d(u,X_i)
\ge |u|
\ge \frac{|\Tf|}{2}.
$$
Moreover, if $u\in \Tf_{X_i}$, then $|u\wedge X_i|=|X_i|$. Hence, assuming $H<|\Tf|/4$, on the event $\mathcal{E}_2$ we obtain that for all
$u\in B_{\ge |\Tf|/2}(X_i,|X_i|)\cup \Tf_{X_i}$,
$$
|u\wedge X_i|>\frac{|\Tf|}{8}.
$$
In particular, on $\mathcal{E}_1 \cap \mathcal{E}_2$, for any $i\neq j \le N$, the sets
$B_{\ge |\Tf|/2}(X_i,|X_i|)\cup \Tf_{X_i}$ and
$B_{\ge |\Tf|/2}(X_j,|X_j|)\cup \Tf_{X_j}$
are disjoint. Therefore, conditionally on $\mathcal{H}$ and on $\mathcal{E}_1 \cap \mathcal{E}_2$, the events $\mathcal{F}_i \cup \mathcal{O}_i$, $i\le N$, are independent. Besides, we have
$$
\sharp B_{\ge |\Tf|/2}(X_i, |X_i|) \le (K+1)K^{|X_i|}
\qquad \text{and} \qquad
\sharp\Tf_{X_i} \le K^{|\Tf|/2+H+1}.
$$
Hence, conditionally on $\mathcal{H}$ and on the event $\mathcal{E}_1 \cap \mathcal{E}_2$, the time until the first clock rings in $B_{\ge |\Tf|/2}(X_i, |X_i|)$,
namely
\[
\min\{\tau_v:\ v\in B_{\ge |\Tf|/2}(X_i, |X_i|)\},
\]
is stochastically lower bounded by an exponential random variable with parameter $(K+1)K^{|X_i|}$. Moreover, the number of clocks that have rung before time $\tau_i$ in $\Tf_{X_i}$ is stochastically upper bounded by a binomial random variable with parameters $\bigl(K^{|\Tf|/2+H+1},\, 1-e^{-\tau_i}\bigr)$. Let $e_h$ denote an exponential random variable with parameter $(K+1)K^{h}$ and let $Y_t$ be a binomial random variable with parameters $\bigl(K^{|\Tf|/2+H+1},\, 1-e^{-t}\bigr)$. We obtain
\begin{eqnarray*}
\P(\mathcal{E}_4\mid \mathcal{H})&\le& 1_{\{\mathcal{E}_1\cap \mathcal{E}_2\}} \prod_{i=1}^N\P(e_{|X_i|}\le \tau_i \mbox{ or } Y_{\tau_i}\ge C \mid \mathcal{H})\\
&=& 1_{\{\mathcal{E}_1\cap \mathcal{E}_2\}} \prod_{i=1}^N\Big(1-\P(e_{|X_i|}> \tau_i \mid \mathcal{H})\P( Y_{\tau_i}<C \mid \tau_i)\Big).\end{eqnarray*}
We write, for a fixed $A>0$,
$$\P( Y_{\tau_i}<C \mid \tau_i)\ge 1_{\{\tau_i\le A K^{-|\Tf|/2}\}}\P(Y_{A K^{-|\Tf|/2}}<C).$$
We set $q_{A,C,|\Tf|}:=\P(Y_{A K^{-|\Tf|/2}}<C)$ so that 
\begin{eqnarray*}
\P(\mathcal{E}_4\mid \mathcal{H})&\le&  1_{\{\mathcal{E}_1\cap \mathcal{E}_2\}} \prod_{i=1}^N\left(1-\exp(-(K+1)K^{|X_i|}\tau_i)q_{A,C,|\Tf|}1_{\{\tau_i\le A K^{-|\Tf|/2}\}}\right).\end{eqnarray*}
Let $m$ denote the largest integer strictly smaller than $|\Tf|/2$, and let $\Tf_m^*$ be the set of vertices of $\Tf\setminus\{o\}$ whose height is at most $m$. We have
$$
\sharp\Tf_m^*=\frac{K(K^{m}-1)}{K-1}.
$$
Note that $X_1,\ldots,X_N$ are $N$ distinct vertices of $\Tf_m^*$ chosen uniformly at random. Moreover, the random variables $(\tau_i)_{i\le N}$ and $(X_i)_{i\le N}$ are independent, and $(\tau_i-\tau_{i-1})_{i\le N}$ are independent exponential random variables with parameters $\sharp\Tf_m^* - i$. We obtain, for any $i\ge 1$ and $h<m$,
$$
\P(|X_i|=m-h)
=
\frac{K^{m-h}}{\sharp\Tf_m^*}
=
\frac{(K-1)K^{m-h}}{K(K^m-1)}
\ge \frac{1}{2K^h}.
$$
Moreover, for any $h_i\ge 1$, we have
\begin{eqnarray*}
\P((|X_1|,\ldots,|X_N|)=(h_1,\ldots,h_N))&\le& \left(\frac{\sharp\Tf_m^*}{\sharp\Tf_m^* - N}\right)^N\prod_{i=1}^N\P(|X_i|=h_i)\\
&\le & \left(1-\frac{N}{K^{m}}\right)^{-N}\prod_{i=1}^N\P(|X_i|=h_i)\\
&\le & \left(1-\frac{N}{K^{|\Tf|/2-1}}\right)^{-N}\prod_{i=1}^N\P(|X_i|=h_i).
\end{eqnarray*}
Thus, we get 
\begin{eqnarray*}
\P(\mathcal{E}_4\mid (\tau_i)_{i\le N})&\le& \sum_{h_1,\ldots,h_N\ge 1} \P((|X_1|,\ldots,|X_N|)=(h_1,\ldots,h_N))\prod_{i=1}^N\left (1-e^{-\tau_i(K+1)K^{h_i}}q_{A,C,|\Tf|}1_{\{\tau_i\le A K^{-|\Tf|/2}\}}\right)\\
&\le &\left(1-\frac{N}{K^{|\Tf|/2-1}}\right)^{-N}\sum_{h_1,\ldots,h_N\ge 0} \prod_{i=1}^N\left (\P(|X_i|=h_i)(1-e^{-\tau_i(K+1)K^{h_i}}q_{A,C,|\Tf|}1_{\{\tau_i\le A K^{-|\Tf|/2}\}})\right)\\
&\le& \left(1-\frac{N}{K^{|\Tf|/2-1}}\right)^{-N}\prod_{i=1}^N \E\left(1-e^{-\tau_i(K+1)K^{|X_i|}}q_{A,C,|\Tf|}1_{\{\tau_i\le A K^{-|\Tf|/2}\}} \mid (\tau_i)_{i\le N}
\right)\\
&=&  \left(1-\frac{N}{K^{|\Tf|/2-1}}\right)^{-N}\prod_{i=1}^N \left(1-\E\left(e^{-\tau_i(K+1)K^{|X_i|}} \mid (\tau_i)_{i\le N}\right)q_{A,C,|\Tf|}1_{\{\tau_i\le A K^{-|\Tf|/2}\}}
\right)\\
&\le&  \left(1-\frac{N}{K^{|\Tf|/2-1}}\right)^{-N}\prod_{i=1}^N \left(1-\left(\sum_{h=0}^{L}\frac{1}{2K^h}e^{-\tau_i(K+1)K^{m-h}}\right) q_{A,C,|\Tf|}1_{\{\tau_i\le A K^{-|\Tf|/2}\}}
\right)
\end{eqnarray*}
valid for any $L<m$ (recall that $m$ is the largest integer strictly smaller than $|\Tf|/2$).
\medskip

We now use that the random variables $(\tau_i-\tau_{i-1})_{i\le N}$ are independent exponential random variables with parameters $\sharp\Tf_m^* - i$. If $i\le \sharp\Tf_m^*/2$, we have
$$\sharp\Tf_m^* - i\ge  \frac{\sharp\Tf_m^*}{2}\ge K^{m-1}\ge K^{|\Tf|/2-2}.$$
Thus, if $N\le K^m/2$, the sequence $(K^{|\Tf|/2-2}(\tau_i-\tau_{i-1}))_{i\le N}$ is stochastically smaller than a sequence $(\varepsilon_i)_{i\le N}$ of i.i.d.\ exponential random variables with parameter 1. Let $\sigma_i=\sum_{k=1}^i\varepsilon_k$.
We get 
\begin{eqnarray*}
\P(\mathcal{E}_4)
&\le&  \left(1-\frac{N}{K^{|\Tf|/2-1}}\right)^{-N}\E\left(\prod_{i=1}^N \left(1-\left(\sum_{h=0}^L\frac{1}{2K^h}e^{-\tau_i(K+1)K^{|\Tf|/2-h}}\right) q_{A,C,|\Tf|}1_{\{\tau_i\le A K^{-|\Tf|/2}\}}
\right)\right)\\
&\le& \left(1-\frac{N}{K^{|\Tf|/2-1}}\right)^{-N}\E\left(\prod_{i=1}^N \left(1-\left(\sum_{h=0}^L\frac{1}{2K^h}e^{-\sigma_i(K+1)K^{2-h}}\right) q_{A,C,|\Tf|}1_{\{\sigma_i\le A K^{-2}\}}
\right)\right).
\end{eqnarray*}
We now use the following lemma, whose proof is postponed until the end of this section.
\begin{lemma}\label{lem:technique}
    Let $\eta>0$. Then, there exist integers $N,L>0$ such that if $(\varepsilon_i)_{i\le N}$ are i.i.d.\ exponential random variables with parameter 1 and $\sigma_i=\sum_{k=1}^i\varepsilon_k$, we have 
    \begin{equation}\label{eq:technique}
        \E\left(\prod_{i=1}^N \left(1-\sum_{h=0}^L\frac{1}{2K^h}e^{-\sigma_i(K+1)K^{2-h}}\right) \right)\le \eta/5.
    \end{equation}    
\end{lemma}
We finish the proof of Proposition \ref{prop:unsommetvide} assuming that Lemma \ref{lem:technique} holds. Fix $\eta>0$ and choose $N,L$ such that \eqref{eq:technique} holds. Since for each $i\le N$, $1_{\{\sigma_i\le A K^{-2}\}}$ tends a.s. to 1 as $A$ tends to infinity, by the dominated convergence theorem, there exists $A$ large enough such that 
\begin{equation*}
        \E\left(\prod_{i=1}^N \left(1-1_{\{\sigma_i\le A K^{-2}\}}\sum_{h=0}^L\frac{1}{2K^h}e^{-\sigma_i(K+1)K^{2-h}}\right) \right)\le \eta/4.
    \end{equation*}
Using \eqref{eq:E2}, fix now $H$ such that 
\begin{equation*}
\P(\mathcal{E}_2^c)\le N\frac{1}{K^{H}}\le \eta/3.
\end{equation*}
Recall that 
$$q_{A,C,|\Tf|}:=\P(Y_{A K^{-|\Tf|/2}}<C) \mbox{ where } Y_{A K^{-|\Tf|/2}}\sim \mbox{Bin}(K^{|\Tf|/2+H+1}, 1-e^{-A K^{-|\Tf|/2}}).$$
Thus, we get that
$$1-q_{A,C,|\Tf|}=\P(Y_{A K^{-|\Tf|/2}}\ge C)\le \frac{\E(Y_{A K^{-|\Tf|/2}})}{C}\le \frac{A K^{H+1}}{C}.$$
This shows that, for fixed $A$ and $H$, we can find $C$ large enough such that $q_{A,C,|\Tf|}$ is close to 1. Using again the dominated convergence theorem, we deduce the existence of a $C$ such that 
$$\E\left(\prod_{i=1}^N \left(1-\left(\sum_{h=0}^L\frac{1}{2K^h}e^{-\sigma_i(K+1)K^{2-h}}\right) q_{A,C,|\Tf|}1_{\{\sigma_i\le A K^{-2}\}}
\right)\right)\le \eta/3,$$
and so, for this choice of $N,H,C$, we obtain $\P(\mathcal{E}_4)\le \eta/2$ if $|\Tf|$ is large enough. Combining this with \eqref{eq:E1} and \eqref{eq:E3}, we conclude that for this choice of $H,C,N$ and for $|\Tf|$ large enough,
\begin{equation*}
 \P(\forall 1\le i\le N,\; \mathcal{A}_i^c)\le \eta.
\end{equation*}
This is exactly the statement of Proposition \ref{prop:unsommetvide}.
\end{proof}

\begin{proof}[Proof of Lemma \ref{lem:technique}]
Let us note that $\sigma_i$ is the sum of $i$ i.i.d.\ exponential random variables; thus $\sigma_i/i$ tends a.s. to 1 as $i$ tends to infinity. Hence, we can find $D>0$ such that 
$$\P(\exists i\ge 1, K^2(K+1)\sigma_i\ge D i)\le \eta/10.$$
With this choice for $D$, we get 
$$ \E\left(\prod_{i=1}^N \left(1-\sum_{h=0}^L\frac{1}{2K^h}e^{-\sigma_i(K+1)K^{2-h}}\right) \right)\le \eta/10+ \prod_{i=1}^N \left(1-\sum_{h=0}^L\frac{1}{2K^h}e^{-DiK^{-h}}\right).$$
Define
\begin{equation*}
       \Pi_{L,N}:= \prod_{i=1}^N \left(1-\sum_{h=0}^L\frac{1}{2K^h}e^{-DiK^{-h}}\right)
    \end{equation*}
and 
\begin{equation*}
      \Pi_{\infty}:= \prod_{i=1}^\infty \left(1-\sum_{h=0}^\infty\frac{1}{2K^h}e^{-DiK^{-h}}\right).
\end{equation*}
Since $\Pi_{L,N}$ tends to $\Pi_\infty$ when $L,N$ tend to infinity, it only remains to show that $\Pi_{\infty}=0$, as this will imply that, for $L$ and $N$ large enough, $\Pi_{L,N}\le \eta/10$, as required. Indeed, we have
$$\log \Pi_\infty = \sum_{i=1}^\infty \log \left(1-\sum_{h=0}^\infty\frac{1}{2K^h}e^{-DiK^{-h}}\right)\le -\sum_{i=1}^\infty \sum_{h=0}^\infty\frac{1}{2K^h} e^{-DiK^{-h}},$$
and    
$$\sum_{h=0}^\infty \sum_{i=1}^\infty\frac{1}{2K^h}e^{-DiK^{-h}}=\sum_{h=0}^\infty \frac{1}{2K^h}\frac{e^{-DK^{-h}}}{1-e^{-DK^{-h}}}=\infty,$$
since $$\lim_{h\to \infty } \frac{1}{2K^h}\frac{e^{-DK^{-h}}}{1-e^{-DK^{-h}}}=\frac{1}{2D}.$$
This yields that $\Pi_\infty=0$, and so $\Pi_{L,N}$ tends to 0 as $N$ and $L$ tend to infinity.
\end{proof}

\subsection{Proof of Proposition \ref{prop:main}}
We first prove a two-window version of Proposition \ref{prop:unsommetvide}: with high probability, one can find a good vertex among the first $N$ candidates and another one among the candidates with indices between $N+1$ and $N'$. The final step, carried out afterwards, will be to show that these two good vertices may be chosen with disjoint descendant subtrees, as required in Proposition \ref{prop:main}.

\begin{prop}\label{prop:deuxcas} Let $\eta>0$. Then, there exist $H,C,N,N',n_0\ge 0$ such that, if $|\Tf|\ge n_0$, then for any choice of vertices $w_1,\ldots,w_C$ of height greater than or equal to $|\Tf|-2H$, starting from $B_0=\{o\}$ and $R_0=\partial \Tf \cup\{w_1,\ldots,w_C\}$, we have
\begin{equation}\label{eq:prop_point1}
 \P(\forall 1\le i\le N,\; \mathcal{A}_i^c)\le \eta.
\end{equation}
\begin{equation}\label{eq:prop_point2}
\P(\forall N<i\le N',\; \mathcal{A}_i^c)\le \eta.
\end{equation}
\end{prop}

\begin{proof}
For $N<N'$, we introduce the following events:
    \begin{eqnarray*}
\mathcal{E}'_1=\mathcal{E}_1(\Tf,N',H,C,R_0) &:=& \{\forall i\neq j \le N', |X_i\wedge X_j|< |\Tf|/8\},\\
\mathcal{E}'_2=\mathcal{E}_2(\Tf,N',H,C,R_0) &:=& \{\forall i \le N', |X_i|> |\Tf|/2-H\},\\
\mathcal{E}'_3=\mathcal{E}_3(\Tf,N',H,C,R_0) &:=& \{\forall i \le N', \forall k \le C, X_i\notin B(w_k,|\Tf|/2+H)\},\\
\mathcal{E}'_4=\mathcal{E}'_4(\Tf,N,N',H,C,R_0) &:=&\{\forall i \le N, \mathcal{A}_i^c\}\cap  {\mathcal{E}'_1} \cap  {\mathcal{E}'_2} \cap  {\mathcal{E}'_3},\\
\mathcal{E}'_5=\mathcal{E}'_5(\Tf,N,N',H,C,R_0) &:=&\{\forall N<i \le N', \mathcal{A}_i^c\}\cap  {\mathcal{E}'_1} \cap  {\mathcal{E}'_2} \cap  {\mathcal{E}'_3}.
\end{eqnarray*}
Note that $\P(\mathcal{E}'_4)\le \P(\mathcal{E}_4)$. Moreover, arguing exactly as in the proof of Proposition \ref{prop:unsommetvide}, we obtain
\begin{eqnarray*}
\P(\mathcal{E}'_5)
&\le&
\left(1-\frac{N'}{K^{|\Tf|/2-1}}\right)^{-N'}
\E\left[
\prod_{i=N+1}^{N'}
\left(
1-
\left(
\sum_{h=0}^{L}\frac{1}{2K^h}
e^{-\sigma_i(K+1)K^{2-h}}
\right)
q_{A,C,|\Tf|}
\,1_{\{\sigma_i\le A K^{-2}\}}
\right)
\right],
\end{eqnarray*}
where
$
\sigma_i=\sum_{k=1}^{i}\varepsilon_k,
$
and $(\varepsilon_k)_{k\le N'}$ are i.i.d.\ exponential random variables with parameter $1$.
Let $\eta>0$. Choose $N$, $N'$, and $L$ with $N'>N$ such that
\begin{equation*}
\E\left(
\prod_{i=1}^{N}
\left(
1-\sum_{h=0}^{L}\frac{1}{2K^h}
e^{-\sigma_i(K+1)K^{2-h}}
\right)
\right)
\le \frac{\eta}{5},
\end{equation*}
and
\begin{equation*}
\E\left(
\prod_{i=N+1}^{N'}
\left(
1-\sum_{h=0}^{L}\frac{1}{2K^h}
e^{-\sigma_i(K+1)K^{2-h}}
\right)
\right)
\le \frac{\eta}{5}.
\end{equation*}

Finally, choosing $A$, $H$, and $C$ large enough exactly as in the proof of Proposition~\ref{prop:unsommetvide}, we obtain that, for this choice of parameters $(N,N',H,C)$, both \eqref{eq:prop_point1} and \eqref{eq:prop_point2} hold whenever $|\Tf|$ is large enough.

\end{proof}

\begin{proof}[Proof of Proposition \ref{prop:main}]
Let $H$, $C$, $N$, $N'$, and $n_0$ be chosen so that Proposition \ref{prop:deuxcas} holds. Then, for every $|\Tf|\ge n_0$, with probability at least $1-2\eta$, there exist indices $i<N<j$ such that both $\mathcal{A}_i$ and $\mathcal{A}_j$ occur. Recall that $X_1,X_2,\ldots$ denote the vertices of $\Tf\setminus\{o\}$ of height smaller than $|\Tf|/2$, ordered according to their coloring times. We have
$$
\P\bigl(\exists\, i<j\le N' :
\Tf_{X_i}\cap\Tf_{X_j}\neq\emptyset\bigr)
\le
N'^2\,\P\bigl(\Tf_{X_1}\cap\Tf_{X_2}\neq\emptyset\bigr)
\le
\frac{|\Tf|\,N'^2}{K^{|\Tf|/2-1}}.
$$
Therefore,
\[
\P\bigl(\forall\, i<j\le N',\ \Tf_{X_i}\cap\Tf_{X_j}=\emptyset\bigr)
\longrightarrow 1
\qquad\text{as } |\Tf|\to\infty.
\]
Combining this estimate with Proposition \ref{prop:deuxcas}, we obtain that, for $|\Tf|$ sufficiently large,
\begin{equation*}
\P\bigl(\exists\, i<j,\ \Tf_{X_i}\cap\Tf_{X_j}=\emptyset
\text{ and }
\mathcal{A}_i\cap\mathcal{A}_j\bigr)
\ge 1-3\eta.
\end{equation*}
This completes the proof of Proposition \ref{prop:main}.
\end{proof}

\section*{Acknowledgements}
The authors warmly thank Nicolas Broutin and Guillaume Chapuy for insightful discussions during the early stages of this work.

\bibliographystyle{plain}
\bibliography{bilbio}
\newpage

\end{document}